\documentclass[letterpaper, 10 pt, conference]{ieeeconf}  

\IEEEoverridecommandlockouts                              

\usepackage[T1]{fontenc}    
\usepackage{hyperref}       
\usepackage{url}            
\usepackage{booktabs}       
\usepackage{amsfonts}       
\usepackage{nicefrac}       
\usepackage{microtype}      
\usepackage{amsmath}
\usepackage{amssymb}
\usepackage{comment}

\def\argmin{\mathop{\rm argmin}}

\usepackage[mathscr]{euscript}
\usepackage{algorithm}
\usepackage{algorithmic}

\newtheorem{theorem}{Theorem}
\newtheorem{corollary}{Corollary}
\newtheorem{lemma}{Lemma}
\newtheorem{definition}{Definition}
\newtheorem{remark}{Remark}
\newtheorem{assumption}{Assumption}
\newtheorem{proposition}{Proposition}

\usepackage{array,multirow,graphicx}
 \usepackage{float}
\usepackage{caption}
\title{\LARGE \bf
Bilevel Distributed Optimization in Directed Networks
}

\author{Farzad Yousefian$^{1}$
\thanks{*The author gratefully acknowledges the support of the National Science Foundation through CAREER grant ECCS-1944500.}
\thanks{$^{1}$Assistant Professor in the School of Industrial Engineering and Management, Oklahoma State University, Stillwater, OK 74078, USA
        {\tt\small (email: farzad.yousefian@okstate.edu)}}%
}

\begin{document}

\maketitle
\thispagestyle{empty}
\pagestyle{empty}

\begin{abstract}
Motivated by emerging applications in wireless sensor networks and large-scale data processing, we consider distributed optimization over directed networks where the agents communicate their information locally to their neighbors to cooperatively minimize a global cost function. We introduce a new unifying distributed constrained optimization model that is characterized as a bilevel optimization problem. This model captures a wide range of existing problems over directed networks including: (i) Distributed optimization with linear constraints; (ii) Distributed unconstrained nonstrongly convex optimization over directed networks. Employing a novel regularization-based relaxation approach and gradient-tracking schemes, we develop an iteratively regularized push-pull gradient algorithm. We establish the consensus and derive new convergence rate statements for suboptimality and infeasibility of the generated iterates for solving the bilevel model. The proposed algorithm and the complexity analysis obtained in this work appear to be new for addressing the bilevel model and also for the two sub-classes of problems. The numerical performance of the proposed algorithm is presented.
\end{abstract}

\section{Introduction}\label{sec:intro}
We consider a class of bilevel distributed optimization problems in directed networks given as follows:
\begin{align}\label{eqn:bilevel_problem}
&\min_{x \in \mathbb{R}^n}\   \sum_{i=1}^m f_i(x) \quad \hbox{s.t.} \  x \in \argmin_{x \in \mathbb{R}^n}\  \sum_{i=1}^m g_i(x),
\end{align}
where we make the following assumptions:
\begin{assumption}\label{assum:problem}
\noindent (a) Functions $f_i:\mathbb{R}^n \to \mathbb{R}$ are $\mu_f$--strongly convex and $L_f$--smooth for all $i$. (b) Functions $g_i:\mathbb{R}^n \to \mathbb{R}$ are convex and $L_g$--smooth for all $i$. (c) The set $\argmin_{x \in \mathbb{R}^n} \textstyle\sum_{i=1}^m g_i(x)$ is nonempty. 
\end{assumption}
Here, $m$ agents cooperatively seek to find among the optimal solutions to the problem $\min_{x \in \mathbb{R}^n}\textstyle\sum_{i=1}^mg_i(x)$, one that minimizes a secondary metric, i.e., $\textstyle\sum_{i=1}^mf_i(x)$. Here, functions $f_i$ and $g_i$ are known locally only by agent $i$ and the cooperation among the agents occurs over a directed network. Given a set of nodes $\mathcal{N}$, a directed graph (digraph) is denoted by $\mathcal{G} =\left(\mathcal{N},\mathcal{E}\right)$ where $\mathcal{E} \subseteq \mathcal{N}\times \mathcal{N}$ is the set of ordered pairs of vertices. For any edge $(i,j)\in\mathcal{E}$, $i$ and $j$ are called parent node and child node, respectively. Graph $\mathcal{G}$ is called {\it strongly connected} if there is a path between the pair of any two different vertices. The digraph induced by a given nonnegative matrix $\mathbf{B}\in \mathbb{R}^{m\times m}$ is denoted by $\mathcal{G}_{\mathbf{B}} \triangleq \left(\mathcal{N}_{\mathbf{B}},\mathcal{E}_{\mathbf{B}}\right)$, where $\mathcal{N}_{\mathbf{B}} \triangleq [m]$ and $(j,i) \in \mathcal{E}_{\mathbf{B}}$ if and only if $B_{ij}>0$. We let $\mathcal{N}_{\mathbf{B}}^{\text{in}}(i)$ and $\mathcal{N}_{\mathbf{B}}^{\text{out}}(i)$ denote the set of parents (in-neighbors) and the set of children (out-neighbors) of vertex $i$, respectively. Also, $\mathcal{R}_\mathbf{B}$ denotes the set of roots of all possible spanning trees in $\mathcal{G}_\mathbf{B}$.

\noindent \textbf{Special cases of the proposed model:} Problem \eqref{eqn:bilevel_problem} provides a unifying mathematical framework capturing several existing problems in the distributed optimization literature. From these, we present two important cases below:\\ 
\noindent (i) Distributed linearly constrained optimization in directed networks: Consider the model given as:
\begin{equation}\label{eqn:lp_cnstr_problem}
\min_{x \in \mathbb{R}^n}\  \textstyle\sum\nolimits_{i=1}^m f_i(x)\quad \hbox{s.t.} \quad   \begin{array}{@{}ll@{}}
    A_ix = b_i \quad \hbox{for all } i \in [m], \\
    x_j \geq 0 \quad \hbox{for } j \in \mathcal{J} \subseteq [n],
  \end{array} 
\end{equation}
where $A_i \in \mathbb{R}^{m_i\times n}$ and $b_i \in \mathbb{R}^{m_i}$ are known parameters. Let problem \eqref{eqn:lp_cnstr_problem} be feasible. Then, by defining for $i \in [m]$, \begin{align}\label{eqn:gi}
g_i(x) :=\tfrac{1}{2} \|A_ix-b_i\|_2^2+\tfrac{1}{2m}\textstyle\sum_{j\in \mathcal{J}}\max\{0,-x_j\}^2,
\end{align} problem \eqref{eqn:lp_cnstr_problem} is equivalent to \eqref{eqn:bilevel_problem} (cf. proof of Corollary \ref{cor:rate_lin_constr}).

\noindent (ii) Distributed unconstrained optimization in the absence of strong convexity: Let us define $f_i(x) : =\|x\|^2_2	/m$. Then, problem \eqref{eqn:bilevel_problem} is equivalent to finding the least $\ell_2$-norm solution of the following canonical distributed unconstrained optimization problem:
\begin{align}\label{eqn:l2_uncnstr_problem}
&\min_{x \in \mathbb{R}^n} \textstyle\sum\nolimits_{i=1}^m g_i(x),
\end{align}
where $g_i$'s are all smooth merely convex (cf. Corollary \ref{cor:rate_unconstr_nonstrongly}). 

\noindent \textbf{Existing theory on distributed optimization in networks:} 
The classical mathematical models, tools, and algorithms for \textit{consensus-based optimization} were introduced and studied as early as the '70s~\cite{Groot74} and '80s~\cite{Tsit84,Tsit86,BertTsitBook}. Of these, in the seminal work of Tsitsiklis~\cite{Tsit84}, it was assumed the agents share a \textit{global (smooth) objective} while their decision component vectors are distributed \textit{locally} over the network. In the past two decades, in light of the unprecedented growth in data and its imperative role in several broad fields such as social networks, biology, and medicine, the theory of distributed and parallel optimization over networks has much advanced. The distributed optimization problems with \textit{local objective functions} were first studied in~\cite{Lopes07,Nedich09}. In this framework, the agents communicate their local information with their neighbors in the network at discrete times to cooperatively minimize the global cost function. Without characterizing the communication rules explicitly, this model can be formulated as $\textstyle\sum\nolimits_{i=1}^m f_i(x)$ subject to $x \in \mathcal{X}$. Here, the local function $f_i$ is known only to the agent $i$ and $\mathcal{X}$ denotes the system constraint set. This modeling framework captures a wide spectrum of decentralized applications in the areas of statistical learning, signal processing, sensor networks, control, and robotics~\cite{Duarte14}. Because of this, in the past decade, there has been a flurry of research focused on the design and analysis of fast and scalable computational methods to address applications in networks. Among these, average-based consensus methods are one of the most studied approaches. Here, the network is characterized with a \textit{stochastic matrix} that is possibly time-varying. The underlying idea is that at a given time, each agent uses this matrix and obtains a weighted-average of its neighbors' local variables.  Then, the update is completed by performing a standard subgradient step for the agent. 
Random projection distributed scheme were developed for both synchronous and asynchronous cases, assuming $\mathcal{X}\triangleq \cap_{i=1}^m\mathcal{X}_i$~\cite{Lee16,Srivastava11}. For the constrained model where $\mathcal{X}$ is easy-to-project, a dual averaging scheme was developed in~\cite{Duchi12}. The algorithm EXTRA~\cite{EXTRA15} and its proximal variant were developed addressing $\mathcal{X}=\mathbb{R}^n$. 
EXTRA is a synchronous and time-invariant scheme and achieves a sublinear and a linear rate of convergence for smooth merely convex and strongly convex problems, respectively. Among many other works such as~\cite{Lobel11,FastDO14}, is the DIGing algorithm~\cite{Diging17} which was the first work achieving a linear convergence rate for unconstrained optimization over a time-varying network. When the graph is directed, a key shortcoming in the weighted-based schemes lies in that the double stochasticity requirement of the weight matrix is impractical. Push-sum protocols were first leveraged in~\cite{NedichAlex15,NedichAlex16,pushi2_2020} to weaken this requirement. Recently, the Push-Pull algorithm equipped with a linear convergence rate was developed in \cite{pushi1_2020} for unconstrained strongly convex problems. Extensions of push-sum algorithms to nonconvex regimes have been developed more recently~\cite{Sun19,Lorenzo16,Touri17}. Other popular distributed optimization schemes are the dual-based methods, such as ADMM-type methods studied in~\cite{BoydAdmm10,Wei13,Wei12,Ling14,Shi14,Aybat17}. Most of these algorithms can address only static and undirected graphs. Moreover, there are only a few works in the literature that can cope with constraints employing primal-dual methods~\cite{Chang14,Nunez15,Chang16,Aybat16}.	

\noindent \textbf{Research gap and contributions:} Despite much advances, the existing models and algorithms for in-network optimization have some shortcomings. For example, the problem is often assumed to be unconstrained, e.g., in Push-DIGing~\cite{Diging17} and Push-Pull \cite{pushi1_2020} algorithms that have been recently developed. Further, the complexity analysis in those algorithms is done under the assumption that the objective function is strongly convex. In this work, we aim at  addressing these shortcomings through considering the bilevel framework \eqref{eqn:bilevel_problem}. Utilizing a novel regularization-based relaxation approach, we develop a new push-pull gradient algorithm where at each iteration, the information of iteratively regularized gradients is pushed to the neighbors, while the information about the decision variable is pulled from the neighbors. We establish the consensus and derive new convergence rate statements for suboptimality and infeasibility of the generated iterates for solving the bilevel model. The proposed algorithm extends \cite{pushi1_2020} to address a class of bilevel problems. The complexity analysis obtained in this work appears to be new and addresses the aforementioned shortcomings. 

\noindent  \textbf{Notation:} For an integer $m$, the set $\{1,\ldots,m\}$ is denoted as $[m]$. A vector $x$ is assumed to be a column vector (unless otherwise noted) and $x^T$ denotes its transpose. We use $\|x\|_2$ to denote the Euclidean
vector norm of $x$. A continuously differentiable function $f: \mathbb{R}^n\rightarrow \mathbb{R}$ is said to be $\mu_f$--strongly
convex if and only if its gradient mapping is $\mu_f$--strongly monotone, i.e.,  $\left(\nabla f(x) -\nabla f(y)\right)^T(x-y)\geq \mu_f\|x-y\|_2^2$ for any $x,y \in \mathbb{R}^n$.  Also, it is said to be $L_f$--smooth if its gradient mapping is Lipschitz continuous with parameter $L_f>0$, i.e., for any $x, y \in\mathbb{R}^n$, we have
$\|\nabla f(x)-\nabla f(y)\|_2\leq L_f\|x-y\|_2$. We use the following definitions:
\begin{align*}
&\mathbf{x} \triangleq [x_1,\ \ldots,\ x_m]^T , \quad \mathbf{y} \triangleq [y_1,\ \ldots,\ y_m]^T\in \mathbb{R}^{m\times n}\\
&{f}(x)  \triangleq \textstyle\sum\nolimits_{i=1}^m f_i(x), \quad \mathbf{f}(\mathbf{x})  \triangleq \textstyle\sum\nolimits_{i=1}^m f_i(x_i), \\
 &\nabla \mathbf{f}(\mathbf{x}) \triangleq [\nabla f_1(x_1),\ \ldots,\ \nabla f_m(x_m)]^T  \in \mathbb{R}^{m\times n}.
\end{align*}
Analogous definitions apply to functions $g$ and $\mathbf{g}$, and mapping $\nabla \mathbf{g}$. Here, $x_i$ denotes the local copy of the decision vector for agent $i$ and $\mathbf{x}$ includes the local copies of all agents. Vector $y_i$ denotes the auxiliary variable for agent $i$ to track the average of regularized gradient mappings. Throughout, we use the following definition of a matrix norm: Given an arbitrary vector norm $\|\cdot\|$, the induced norm of a matrix $W \in \mathbb{R}^{m \times n}$ is defined as $\|\mathbf{W}\|\triangleq \left\|\left[\left\|\mathbf{W}_{\bullet 1}\right\|,\ldots,\left\|\mathbf{W}_{\bullet n}\right\|\right]\right\|_2$.
\begin{remark}\label{rem:norms}
Under the above definition of matrix norm, it can be seen we have $\|\mathbf{A}\mathbf{x}\| \leq \|\mathbf{A}\|\|\mathbf{x}\|$ for any $\mathbf{A} \in \mathbb{R}^{m \times m}$ and $\mathbf{x} \in \mathbb{R}^{m \times p}$. Also, for any $a \in \mathbb{R}^{m \times 1}$ and $x \in \mathbb{R}^{1 \times n}$, we have $\|ax\|=\|a\|\|x\|_2$.
\end{remark}

\section{Algorithm outline}\label{sec:alg}
To solve the model~\eqref{eqn:bilevel_problem} in directed networks, due to the presence of the inner-level optimization constraints, Lagrangian duality does not seem applicable. Overcoming this challenge calls for new relaxation rules that can tackle the inner-level constraints. We consider a regularization-based relaxation rule. To this end, motivated by the recent success of so-called \textit{iteratively regularized (IR)} algorithms in centralized regimes \cite{FarzadSIOPT20,FarzadMathProg17,FarzadHarshalACC19,FarzadMostafaACC19,HarshalFarzadOptVI2021}, we develop Algorithm \ref{algorithm:IR-push-pull}. Core to the IR framework is the philosophy that the regularization parameter $\lambda_k$ is updated after every step within the algorithm. Here, each agent holds a local copy of the global variable $x$, denoted by $x_{i,k}$, and an auxiliary variable $y_{i,k}$ is used to track the average of a regularized gradient. At each iteration, each agent $i$ uses the $i$th row of two matrices $\mathbf{R} =[R_{ij}]\in\mathbb{R}^{m\times m}$ and $\mathbf{C} =[C_{ij}]\in\mathbb{R}^{m\times m}$ to update vectors $x_{i,k}$ and $y_{i,k}$, respectively. Below, we state the main assumptions on the these two \textit{weight mixing} matrices. 
\begin{assumption}\label{assum:RC}
\noindent (a) The matrix $\mathbf{R}$ is nonnegative, with a strictly positive diagonal, and is row-stochastic, i.e., $\mathbf{R} \mathbf{1} =\mathbf{1}$. 
\noindent (b) The matrix $\mathbf{C}$ is nonnegative, with a strictly positive diagonal, and is column-stochastic, i.e., $ \mathbf{1} ^T\mathbf{C}=\mathbf{1}^T$. 
\noindent (c) The induced digraphs $\mathcal{G}_{\mathbf{R}}$ and  $\mathcal{G}_{\mathbf{C}^T}$ satisfy $\mathcal{R}_{\mathbf{R}}\cap \mathcal{R}_{\mathbf{C}^T}\neq \emptyset$. 
\end{assumption}

\begin{algorithm}
  \caption{Iteratively Regularized Push-Pull}\label{algorithm:IR-push-pull}
    \begin{algorithmic}[1]
    \STATE\textbf{Input:} For all $i \in [m]$, agent $i$ sets step-size $\gamma_{i,0} \geq 0$, pulling weights $R_{ij} \geq 0$ for all $j \in \mathcal{N}_{\mathbf{R}}^{\text{in}}(i)$, pushing weights $C_{ij} \geq 0$ for all $j \in \mathcal{N}_{\mathbf{C}}^{\text{out}}(i)$,  an arbitrary initial point $x_{i,0} \in \mathbb{R}^n$ and $y_{i,0}:=\nabla g_i(x_{i,0})+\lambda_0\nabla f_i(x_{i,0})$; 
    \FOR {$k=0,1,\ldots$}
    		 \STATE For all $i \in [m]$, agent $i$ receives (pulls) the vector $x_{j,k}-\gamma_{j,k}y_{j,k}$ from each agent $j \in \mathcal{N}_{\mathbf{R}}^{\text{in}}(i)$, sends (pushes) $C_{\ell i}y_{i,k}$ to each agent $\ell \in \mathcal{N}_{\mathbf{C}}^{\text{out}}(i)$, and does the following updates:
 $\begin{aligned}& x_{i,k+1} := \textstyle\sum\nolimits_{j=1}^m R_{ij}\left(x_{j,k}-\gamma_{j,k}y_{j,k}\right)\\
 &y_{i,k+1} :=  \textstyle\sum\nolimits_{j=1}^m C_{ij}y_{j,k}+\nabla g_i(x_{i,k+1})\\
      &+\lambda_{k+1}\nabla f_i(x_{i,k+1})
      -\nabla g_i(x_{i,k})-\lambda_{k}\nabla f_i(x_{i,k});\end{aligned}$
   \ENDFOR
   \end{algorithmic}
\end{algorithm}
 Assumption \ref{assum:RC} does not require the strong condition of a doubly stochastic matrix for communication in a directed network. In turn, utilizing a push-pull protocol and in a similar fashion to the recent work~\cite{pushi1_2020}, it only entails a row stochastic $\mathbf{R}$ and a column stochastic matrix $\mathbf{C}$. An example is as follows where agent $i$ chooses scalars $r_i,c_i>0$ and sets $R_{i,j} := 1/\left(\left|\mathcal{N}_{\mathbf{R}}^{\text{in}}(i)\right|+r_i\right)$ for $j \in \mathcal{N}_{\mathbf{R}}^{\text{in}}(i)$, $R_{i,i} :=r_i/\left(\left|\mathcal{N}_{\mathbf{R}}^{\text{in}}(i)\right|+r_i\right)$, $C_{\ell, i} :=1/\left(\left|\mathcal{N}_{\mathbf{C}}^{\text{out}}(i)\right|+c_i\right)$ for $\ell \in \mathcal{N}_{\mathbf{C}}^{\text{out}}(i)$, $C_{i, i} :=c_i/\left(\left|\mathcal{N}_{\mathbf{C}}^{\text{out}}(i)\right|+c_i\right)$, and $0$ otherwise.
Note that Assumption \ref{assum:RC}(c) is weaker than imposing strong connectivity on $\mathcal{G}_{\mathbf{R}}$ and $\mathcal{G}_{\mathbf{C}}$. The update rules in Algorithm \ref{algorithm:IR-push-pull} can be compactly represented as the following:
\begin{align}
\mathbf{x}_{k+1} :=& \mathbf{R}\left(\mathbf{x}_k-\boldsymbol{\gamma}_k\mathbf{y}_k\right),\label{alg:IRPP_compact1}\\
\mathbf{y}_{k+1} := &\mathbf{C}\mathbf{y}_k+\nabla \mathbf{g}(\mathbf{x}_{k+1})+ \lambda_{k+1}\nabla \mathbf{f}(\mathbf{x}_{k+1}) \notag\\
&-\nabla \mathbf{g}(\mathbf{x}_k)  -\lambda_k\nabla \mathbf{f}(\mathbf{x}_k)\label{alg:IRPP_compact2},
\end{align}
where $\boldsymbol{\gamma}_k\geq 0$ is defined as $\boldsymbol{\gamma}_k \triangleq \text{diag}\left(\gamma_{1,k},\ldots,\gamma_{m,k}\right)$.

\section{Preliminaries of convergence analysis}
Under Assumption \ref{assum:RC}, there exists a unique nonnegative left eigenvector $u \in \mathbb{R}^m$ such that $u^T\mathbf{R} = u^T$ and $u^T\mathbf{1} =m$. Similarly, there exists a unique nonnegative right eigenvector $v \in \mathbb{R}^m$ such that $\mathbf{C}v = v$ and $\mathbf{1}^Tv =m$ (cf. Lemma 1 in \cite{pushi1_2020}). Throughout, we use the following definitions 
\begin{definition}\label{eqn:defs}
For $k\geq 0$ and the regularization parameter $\lambda_k>0$, let $x^* \triangleq \argmin_{x \in \argmin g(x)}\{ f(x)\} \in \mathbb{R}^{1\times n}$, $x_{\lambda_k}^*\triangleq \argmin_{x \in \mathbb{R}^n}\left\{ g(x)+\lambda_k f(x)\right\} \in \mathbb{R}^{1\times n}$. We define the mapping $\mathbf{G}_k(\mathbf{x})\triangleq \nabla \mathbf{g}\left(\mathbf{x}\right)+\lambda_k \nabla \mathbf{f}\left(\mathbf{x}\right)\in \mathbb{R}^{m\times n}$, and functions $G_k(\mathbf{x})\triangleq \tfrac{1}{m}\mathbf{1}^T\mathbf{G}_k(\mathbf{x})\in \mathbb{R}^{1\times n}$, $\mathscr{G}_k(x) \triangleq G_k\left(\mathbf{1}x^T\right)\ \in \mathbb{R}^{1\times n}$, $\bar g_k \triangleq \mathscr{G}_k(\bar x_k)\ \in \mathbb{R}^{1\times n}$ where $\bar x_k \triangleq \tfrac{1}{m}u^T\mathbf{x}_k\ \in \mathbb{R}^{1\times n}$. We let $L_k \triangleq L_g+\lambda_kL_f$ and $\bar y_k \triangleq \tfrac{1}{m}\mathbf{1}^T\mathbf{y}_k\ \in \mathbb{R}^{1\times n}$. Lastly, we define $\Lambda_k \triangleq \left|1- \tfrac{\lambda_{k+1}}{\lambda_{k}}\right|$.
\end{definition}

Here, $x^*$ denotes the optimal solution of problem \eqref{eqn:bilevel_problem} and $x^*_\lambda$ is defined as the optimal solution to a regularized problem. Note that the strong convexity of $g(x)+\lambda_kf(x)$ implies that $x^*_{\lambda_k}$ exists and is a unique vector (cf. Proposition 1.1.2 in~\cite{Bertsekas2016}). Also, under Assumption \ref{assum:problem}, the set $\argmin g(x)$ is closed and convex. As such, from the strong convexity of $f$ and invoking Proposition 1.1.2 in~\cite{Bertsekas2016} again, we conclude that $x^*$ also exists and is a unique vector. The sequence $\{x^*_{\lambda_k}\}$ is the so-called \textit{Tikhonov trajectory} and plays a key role in the convergence analysis (cf. \cite{facchinei02finite}). The mapping $\mathbf{G}_k(\mathbf{x})$ denotes the regularized gradient matrix. The vector $\bar x_k$ holds a weighted average of the local copies of the agent's iterates. Next, we consider a family of update rules for the sequences of the step-size and the regularization parameter under which the convergence and rate analysis can be performed. 
\begin{assumption}[Update rules]\label{assum:update_rules}
Assume the step-size $\boldsymbol{\gamma}_k$ and the regularization parameter $\lambda_k$ are updated satisfying: $\hat \gamma_k :=\tfrac{\hat \gamma_0}{(k+1)^a}$ and $\lambda_k :=\tfrac{\lambda_0}{(k+1)^b}$ where $\hat \gamma_k \triangleq \max_{j\in [m]}\gamma_{j,k}$ for $k\geq 0$, and $a$ and $b$ satisfy the following conditions: $0<b<a<1$ and $a+b<1$. Also, let $\alpha_k \geq \theta \hat \gamma_k$ for $k\geq 0$ for some $\theta>0$, where $\alpha_k  \triangleq  \tfrac{1}{m}u^T\boldsymbol{\gamma}_k\nu$.
\end{assumption}
The constant $\theta$ in Assumption \ref{assum:update_rules} measures the size of the range within which the agents in $\mathcal{R}_\mathbf{R}\cap \mathcal{R}_{\mathbf{C}^T}$ select their stepsizes. The condition $\alpha_k \geq \theta \hat \gamma_k$ is satisfied in many cases including the case where all the agents choose strictly positive stepsizes (see Remark 4 in~\cite{pushi1_2020} for more details). In the following lemma, we list some of the main properties of the update rules in Assumption \ref{assum:update_rules} that will be used in the  analysis. 
\begin{lemma}[Properties of the update rules]\label{lem:update_rules_props}
Under Assumption \ref{assum:update_rules}, we have: $\{\lambda_k\}_{k=0}^\infty$ is a decreasing strictly positive sequence satisfying $\lambda_k \to 0$, $\tfrac{\Lambda_k}{\lambda_k} \to 0$,  $\Lambda_{k+1}\leq \Lambda_{k}$ for all $k\geq 0$, $\Lambda_{k-1} \leq \tfrac{1}{k+1}$ for $k\geq 1$, where $\Lambda_k$ is given by Def. \ref{eqn:defs}. Also, $\{\hat \gamma_k\}_{k=0}^\infty$ is a decreasing strictly positive sequence such that $\hat \gamma_k \to 0$ and $\tfrac{\hat \gamma_k}{\lambda_k} \to 0 $. Moreover, for any scalar $\tau>0$, there exists an integer $K_{\tau}$ such that $\tfrac{(k+1)\hat\gamma_k\lambda_k}{k\hat\gamma_{k-1}\lambda_{k-1}}\leq 1+\tau\hat \gamma_k\lambda_k$ for all $k\geq K_{\tau}$.
\end{lemma}
\begin{proof}
Recall that $\hat \gamma_k=\tfrac{\hat \gamma_0}{(k+1)^a}$ and $\lambda_k=\tfrac{ \lambda_0}{(k+1)^b}$ where $0<b<a<1$ and $a+b<1$. Consequently, $\{\hat \gamma_k\}_{k=0}^\infty$ and $\{\lambda_k\}_{k=0}^\infty$ are strictly positive and decreasing sequences and $\hat \gamma_k \to 0$, $\lambda_k \to 0$, and $\tfrac{\hat \gamma_k}{\lambda_k} \to 0 $.  Next, we show that $\Lambda_{k-1} \leq \tfrac{1}{k+1}$ for $k\geq 1$. From Def. \ref{eqn:defs} and that $\lambda_{k}\leq \lambda_{k-1}$, for any $k\geq 1$ we have:
\begin{align}\label{eqn:Lambda_open}
\Lambda_{k-1} =1-\tfrac{\lambda_k}{\lambda_{k-1}}=1-\tfrac{\lambda_0(k+1)^{-b}}{\lambda_0k^{-b}}=1-\left(1-\tfrac{1}{k+1}\right)^b.
\end{align}
From $0<b<a$ and $a+b<1$, we have $b<0.5$. This implies that $\left(1-\tfrac{1}{k+1}\right)^b\geq \left(1-\tfrac{1}{k+1}\right)^{0.5}$. Combining this relation with \eqref{eqn:Lambda_open}, we have:
\begin{align*}
\Lambda_{k-1} &\leq 1-\left(1-\tfrac{1}{k+1}\right)^{0.5}=\tfrac{1-\left(1-\tfrac{1}{k+1}\right)}{1+\sqrt{1-\tfrac{1}{k+1}}}= \tfrac{\tfrac{1}{k+1}}{1+\sqrt{1-\tfrac{1}{k+1}}}\\
&\leq \tfrac{1}{k+1},
\end{align*}
where the last inequality is implied from $k\geq 1$. Next, we show $\Lambda_{k+1}\leq \Lambda_k$ for all $k\geq 0$. From \eqref{eqn:Lambda_open}, we have:
\begin{align*}
\Lambda_{k+1} =1-\left(1-\tfrac{1}{k+3}\right)^b\leq1-\left(1-\tfrac{1}{k+2}\right)^b = \Lambda_{k}.
\end{align*}
Next, we show that for any scalar $\tau>0$, there exists an integer $K_{\tau}$ such that $\tfrac{(k+1)\hat\gamma_k\lambda_k}{k\hat\gamma_{k-1}\lambda_{k-1}}\leq 1+\tau\hat \gamma_k\lambda_k$ for all $k\geq K_{\tau}$. It suffices to show $\lim_{k\to \infty}\left(\tfrac{(k+1)\hat\gamma_k\lambda_k}{k\hat\gamma_{k-1}\lambda_{k-1}}-1\right)\tfrac{1}{\hat \gamma_k \lambda_k}=0$. From update formulas of $\hat \gamma_k$ and $\lambda_k$, we have:
\begin{align*}
&\left(\tfrac{(k+1)\hat\gamma_k\lambda_k}{k\hat\gamma_{k-1}\lambda_{k-1}}-1\right)\tfrac{1}{\hat \gamma_k \lambda_k}=\left(\tfrac{(k+1)}{k}\left(\tfrac{k}{k+1}\right)^{a+b}-1\right)\tfrac{(k+1)^{a+b}}{\hat \gamma_0 \lambda_0}\\
&\leq \left(\left(1+\tfrac{1}{k}\right)^{1-a-b}-1\right)\tfrac{2k^{a+b}}{\hat \gamma_0 \lambda_0}\leq 
\left(1+\tfrac{1}{k}-1\right)\tfrac{2k^{a+b}}{\hat \gamma_0 \lambda_0}\\
&=\tfrac{2}{\hat \gamma_0 \lambda_0k^{1-a-b}}.
\end{align*}
Therefore, $\lim_{k\to \infty}\left(\tfrac{(k+1)\hat\gamma_k\lambda_k}{k\hat\gamma_{k-1}\lambda_{k-1}}-1\right)\tfrac{1}{\hat \gamma_k \lambda_k}=0$, implying the existence of the specified integer $K_{\tau}$.
\end{proof}
Next, we present some key properties of the regularized sequence $\{x^*_{\lambda_k}\}$ that will be used in the rate analysis. 
\begin{lemma}[Properties of Tikhonov trajectory]\label{lemma:IR-props}
Let Assumptions \ref{assum:problem} and \ref{assum:update_rules} hold and $x^*_{\lambda_k}$ be given by Def. \ref{eqn:defs}. Then, we have: (i) The sequence $\{x^*_{\lambda_k}\}$ converges to the unique solution of problem \eqref{eqn:bilevel_problem}, i.e., $x^*$. (ii) There exists a scalar $M>0$ such that for any $k \geq 1$, we have $
\left\|x^*_{\lambda_k}-x^*_{\lambda_{k-1}}\right\|_2 \leq \tfrac{M}{\mu_f}\Lambda_{k-1}$.
\end{lemma}
\begin{proof}
\noindent \textbf{(a)} From Assumption \ref{assum:problem}, the function $g(x)+\lambda_kf(x)$ is $(\lambda_k\mu_f)$-strongly convex. Since $x^*_{\lambda_k}$ is the minimizer of this regularized function, we have:
	\begin{align}\label{eqn:x_*_x_lambda_relation0}
	g(x^*)+\lambda_kf(x^*)\geq g(x^*_{\lambda_k})+\lambda_kf(x^*_{\lambda_k})+\tfrac{\lambda_k\mu_f\left\|x^*-x^*_{\lambda_k}\right\|^2_2}{2}.
	\end{align} 
Also, the definition of $x^*$ in Def. \ref{eqn:defs} implies that $x^*\in \argmin g(x)$ and so $g(x^*_{\lambda_k})\geq 	g(x^*)$. From the preceding two inequalities, we obtain $
f(x^*) \geq f\left(x^*_{\lambda_k}\right)$ for all $k \geq 0$.  From this and the coercivity of $f$ (due to its strong convexity), the sequence $\{x^*_{\lambda_k}\}$ is bounded implying that it must have at least one limit point. Let $\{x^*_{\lambda_k}\}_{k \in \mathcal{K}}$ be an arbitrary subsequence such that $\lim_{k \to \infty, \ k \in \mathcal{K}}x^*_{\lambda_k} = \hat x$. Next, we show that $\hat x$ is a feasible solution to problem \eqref{eqn:bilevel_problem}. Consider relation \eqref{eqn:x_*_x_lambda_relation0}. Taking the limit from both sides of \eqref{eqn:x_*_x_lambda_relation0}, using the continuity of $g$ and $f$, invoking $\lim_{k \to \infty} \lambda_k =0$, and neglecting the nonnegative term, we obtain that:
\begin{align*}
g\left(x^* \right)  \geq  g\left(\lim_{k\rightarrow \infty, \ k \in \mathcal{K}} x^*_{\lambda_k} \right) = g(\hat x),
\end{align*}
implying that $\hat x$ is a minimizer of $g(x)$ and so it is a feasible solution to problem \eqref{eqn:bilevel_problem}. Next, we show that $\hat x$ is the optimal solution of \eqref{eqn:bilevel_problem}. Taking the limit from both sides of $
f(x^*) \geq f\left(x^*_{\lambda_k}\right)$ and using the continuity of $f$ we obtain:
\begin{align*}
f\left(x^* \right)  \geq  f\left(\lim_{k\rightarrow \infty, \ k \in \mathcal{K}} x^*_{\lambda_k} \right) = f(\hat x).
\end{align*}
 Hence, from the uniqueness of $x^*$ we conclude that all limit points of the Tikhonov trajectory are identical to $x^*$. Therefore, $\lim_{k\to \infty}x^*_{\lambda_k}=x^*$.

		\noindent \textbf{(b)} If $x^*_{\lambda_k} = x^*_{\lambda_{k-1}}$, the desired relation holds. Consider the case where $x^*_{\lambda_k} \neq x^*_{\lambda_{k-1}}$ for $k\geq 1$. Writing the optimality conditions for $x^*_{\lambda_k}$  and $x^*_{\lambda_{k-1}}$, for any $x,y \in \mathbb{R}^n$, we have:
	\begin{align*}
	&g(x)+\lambda_{k-1}f(x)\geq g\left(x^*_{\lambda_{k-1}}\right)+\lambda_{k-1}f\left(x^*_{\lambda_{k-1}}\right)\\
	&+\tfrac{\lambda_{k-1}\mu_f}{2}\left\|x-x^*_{\lambda_{k-1}}\right\|^2_2.
	\\	
	&g(y)+\lambda_kf(y)\geq g\left(x^*_{\lambda_k}\right)+\lambda_kf\left(x^*_{\lambda_k}\right)+\tfrac{\lambda_k\mu_f}{2}\left\|y-x^*_{\lambda_k}\right\|^2_2.
	\end{align*} 
	By substituting $x:=x^*_{\lambda_k}$ and $y:=x^*_{\lambda_{k-1}}$ and  adding the resulting two relations together, we have: 
		\begin{align*}
	&\left(\lambda_{k-1}-\lambda_{k}\right)\left(f\left(x^*_{\lambda_k}\right)-f\left(x^*_{\lambda_{k-1}}\right)\right)\\
	&\geq \tfrac{\mu_f\left(\lambda_k+\lambda_{k-1}\right)}{2}\left\|x^*_{\lambda_{k-1}}-x^*_{\lambda_k}\right\|^2_2. 
	\end{align*}
From the convexity property of $f$, we have that: 
\begin{align*}
f\left(x^*_{\lambda_k}\right)-f\left(x^*_{\lambda_{k-1}}\right)	\leq \left(x^*_{\lambda_k}-x^*_{\lambda_{k-1}}\right)^T\nabla f\left(x^*_{\lambda_{k}}\right).
\end{align*}	
	From the last two inequalities and $\lambda_{k-1}>\lambda_k$, we obtain:
		\begin{align*}
 \tfrac{\mu_f\lambda_{k-1}}{2}\left\|x^*_{\lambda_{k-1}}-x^*_{\lambda_k}\right\|^2_2 &\leq \left(\lambda_{k-1}-\lambda_{k}\right)\left(x^*_{\lambda_k}-x^*_{\lambda_{k-1}}\right)^T\\
 &\nabla f\left(x^*_{\lambda_{k}}\right). 
	\end{align*}
Dividing both sides of the by $\lambda_{k-1}$ and using Cauchy-Schwarz inequality, we have:
		\begin{align*}
 \left\|x^*_{\lambda_{k-1}}-x^*_{\lambda_k}\right\|^2_2 &\leq \tfrac{2}{\mu_f} \left(1-\tfrac{\lambda_{k}}{\lambda_{k-1}}\right)\left\|x^*_{\lambda_k}-x^*_{\lambda_{k-1}}\right\|_2\\
 &\left\|\nabla f\left(x^*_{\lambda_{k}}\right)\right\|_2. 
	\end{align*}	
	Note that since $x^*_{\lambda_k} \neq x^*_{\lambda_{k-1}}$, we have $\left\|x^*_{\lambda_k}-x^*_{\lambda_{k-1}}\right\|_2 \neq 0$. Thus, we obtain: 
			\begin{align*}
 \left\|x^*_{\lambda_{k-1}}-x^*_{\lambda_k}\right\|_2 &\leq \tfrac{2}{\mu_f} \left(1-\tfrac{\lambda_{k}}{\lambda_{k-1}}\right)\left\|\nabla f\left(x^*_{\lambda_{k}}\right)\right\|_2 \\
 &=\tfrac{2\Lambda_{k-1}}{\mu_f} \left\|\nabla f\left(x^*_{\lambda_{k}}\right)\right\|_2 . 
	\end{align*}	
From part (a), $\{x^*_{\lambda_t}\}_{t=0}^\infty$ is a bounded sequence. Thus, there exists a compact ball $\mathcal{X} \subset \mathbb{R}^n$ such that $\{x^*_{\lambda_t}\}_{t=0}^\infty \subseteq \mathcal{X} $. From continuity of the mapping $\nabla f$, there exists a constant $M>0$ such that $
	2\left\|\nabla f\left(x^*_{\lambda_k}\right)\right\|_2\leq M$ for all $k \geq 0$. Combining the preceding two relations, we obtain the desired inequality. 
\end{proof}
In the following, we state the properties of the regularized maps to be used in finding error bounds in the next section. 
\begin{lemma}\label{lemma:grad_track_props}
Consider Algorithm \ref{algorithm:IR-push-pull}. Let Assumptions \ref{assum:problem} and \ref{assum:RC} hold. For any $k\geq 0$, mappings $G_k$, $\mathscr{G}_k$, and $\bar g_k$ given by Def. \ref{eqn:defs} satisfy the following relations: (i) We have that $\bar y_k = G_k(\mathbf{x}_k)$. (ii) We have $\mathscr{G}_k\left(x^*_{\lambda_k}\right) = 0$. (iii)  The mapping $\mathscr{G}_k(x)$ is  $(\lambda_k\mu_f)$-strongly monotone and Lipschitz continuous with parameter $L_k$. (iv) We have $\|\bar y_k - \bar g_k \|_2\leq \tfrac{L_k}{\sqrt{m}}\left\|\mathbf{x}_k-\mathbf{1}\bar x_k\right\|_2$ and $\|\bar g_k \|_2\leq L_k\|\bar x_k-x^*_{\lambda_k}\|_2$.
\end{lemma}
\begin{proof}
\noindent \textbf{(i)} Multiplying both sides of \eqref{alg:IRPP_compact2} by $\tfrac{1}{m}\mathbf{1}^T$ and from the definitions of $\mathbf{G}_k$ and $G_k$ in Def. \ref{eqn:defs}, we obtain: 
\begin{align*}
\bar y_{k+1}&=\tfrac{1}{m}\mathbf{1}^T\mathbf{y}_{k} +  \tfrac{1}{m}\mathbf{1}^T\mathbf{G}_{k+1}(\mathbf{x}_{k+1}) - \tfrac{1}{m}\mathbf{1}^T\mathbf{G}_{k}(\mathbf{x}_{k})\\
&=\bar y_{k} +G_{k+1}(\mathbf{x}_{k+1})-G_{k}(\mathbf{x}_{k}),
\end{align*}
where we used $\mathbf{1}^T\mathbf{C} = \mathbf{1}^T$.
From Algorithm \ref{algorithm:IR-push-pull}, we have $\mathbf{y}_0 : = \nabla \mathbf{g}(\mathbf{x}_0) + \lambda_0 \nabla \mathbf{f}(\mathbf{x}_0)=\mathbf{G}_0(\mathbf{x}_0)$, implying that $\bar y_{0}=G_{0}(\mathbf{x}_{0})$. From the two preceding relations, we obtain that $\bar y_k = G_k\left(\mathbf{x}_k\right)$.

\noindent \textbf{(ii, iii)}
From Def. \ref{eqn:defs}, we have that $G_k(\mathbf{x}) = \tfrac{1}{m}\textstyle\sum_{i=1}^m\left(\nabla g_i\left(x_{i,k}\right)+\lambda_k\nabla f_i\left(x_{i,k}\right)\right)$. Thus, from the definition of $\mathscr{G}_k$ we obtain that $\mathscr{G}_k(x) = G_k\left(\mathbf{1}x^T\right) =  \tfrac{1}{m}\textstyle\sum_{i=1}^m\left(\nabla g_i\left(x\right)+\lambda_k\nabla f_i\left(x\right)\right) = \tfrac{1}{m}\left(\nabla g(x) +\lambda_k \nabla f(x)\right)$. Thus, from the definition of $x^*_{\lambda_k}$ in Def. \ref{eqn:defs}, we obtain $\mathscr{G}_k\left(x^*_{\lambda_k}\right) = 0$. Also, from Assumption \ref{assum:problem}, we conclude that  $\mathscr{G}_k(x)$ is a $(\lambda_k\mu_f)$-strongly monotone mapping and Lipschitz continuous with parameter $L_k \triangleq L_g+\lambda_kL_f$ for $k\geq 0$. 

\noindent \textbf{(iv)} For any $\mathbf{u},\mathbf{v} \in \mathbb{R}^{m\times n}$, with $u_i,v_i \in \mathbb{R}^n$ denoting the $i^{\text{th}}$ row of $\mathbf{u},\mathbf{v}$, respectively, we have:
\begin{align*}
&\left\|G_k(\mathbf{u}) - G_k(\mathbf{v})\right\|_2 = \left\|\tfrac{1}{m}\mathbf{1}^T\left(\nabla \mathbf{g}(\mathbf{u})  +\lambda_k\nabla \mathbf{f}(\mathbf{u})\right) \right.\\
&\left.- \tfrac{1}{m}\mathbf{1}^T\left(\nabla \mathbf{g}\left(\mathbf{v}\right)+\lambda_k \nabla \mathbf{f}\left(\mathbf{v}\right)\right)\right\|_2\\
&  \leq \tfrac{1}{m}\left\|\textstyle\sum_{i=1}^m\nabla g_i(u_{i})-\textstyle\sum_{i=1}^m \nabla g_i\left(v_i\right)\right\|_2\\
&+\tfrac{\lambda_k}{m}\left\|\textstyle\sum_{i=1}^m\nabla f_i(u_i)-\textstyle\sum_{i=1}^m \nabla f_i\left(v_i\right)\right\|_2\\
& \leq \tfrac{1}{m}\textstyle\sum_{i=1}^m\left(\left\|\nabla g_i(u_i)- \nabla g_i\left(v_i\right)\right\|_2\right.\\
&\left.+\lambda_k\left\|\nabla f_i(u_i)- \nabla f_i\left(v_i\right)\right\|_2\right)\\
&\leq \tfrac{1}{m}\textstyle\sum_{i=1}^m\left(L_{g}\|u_i-v_i\|_2+\lambda_kL_{f}\|u_i-v_i\|_2\right) \\
&\leq \tfrac{L_k}{m}\textstyle\sum_{i=1}^m\|u_i-v_i\|_2\leq \tfrac{L_k}{\sqrt{m}}\|\mathbf{u}-\mathbf{v}\|_2.
\end{align*}
Consequently, we obtain $\|\bar y_k - \bar g_k\|_2 = \left\|G_k(\mathbf{x}_k)-G_k(\mathbf{1}\bar x_k)\right\|_2\leq \tfrac{L_k}{\sqrt{m}}\|\mathbf{x}_k-\mathbf{1}\bar x_k\|_2$. Also, using the Lipschitzian property of $\mathscr{G}_k$ in part (ii) and $\mathscr{G}_k\left(x^*_{\lambda_k}\right) = 0$, we obtain: \begin{align*}\|\bar g_k\|_2&=\left\|\mathscr{G}_k(\bar x_k)\right\|_2= \left\|\mathscr{G}_k(\bar x_k)-\mathscr{G}_k\left(x^*_{\lambda_k}\right)\right\|_2\\
&\leq L_k\|\bar x_k-x^*_{\lambda_k}\|_2.
\end{align*}
\end{proof}
We state the following result from \cite{pushi1_2020} introducing two matrix norms induced by matrices $\mathbf{R}$ and $\mathbf{C}$.
\begin{lemma}[cf. Lemma 4 and Lemma 6 in \cite{pushi1_2020}]\label{lemma:norms_relations}
Let Assumption \ref{assum:RC} hold. Then: (i) There exist matrix norms $\|\cdot\|_{\mathbf{R}}$ and  $\|\cdot\|_{\mathbf{C}}$ such that for $\sigma_{\mathbf{R}}\triangleq \left\|\mathbf{R}-\tfrac{\mathbf{1}u^T}{m}\right\|_{\mathbf{R}}$ and $\sigma_{\mathbf{C}}\triangleq \left\|\mathbf{C}-\tfrac{\mathbf{1}v^T}{m}\right\|_{\mathbf{C}}$ we have that $\sigma_{\mathbf{R}}<1$ and $\sigma_{\mathbf{C}}<1$. (ii) There exist scalars $\delta_{\mathbf{R},2},\delta_{\mathbf{C},2},\delta_{\mathbf{R},\mathbf{C}},\delta_{\mathbf{C},\mathbf{R}}>0$ such that for any $W \in \mathbb{R}^{m \times n}$, we have $\|W\|_{\mathbf{R}}\leq \delta_{\mathbf{R},2}\|W\|_2$, $\|W\|_{\mathbf{C}}\leq \delta_{\mathbf{C},2}\|W\|_2$, $\|W\|_{\mathbf{R}}\leq \delta_{\mathbf{R},\mathbf{C}}\|W\|_\mathbf{C}$, $\|W\|_{\mathbf{C}}\leq \delta_{\mathbf{C},\mathbf{R}}\|W\|_\mathbf{R}$, $\|W\|_{2}\leq \|W\|_\mathbf{R}$, and $\|W\|_{2}\leq \|W\|_\mathbf{C}$.
\end{lemma}

\section{Convergence and rate analysis}
We analyze the convergence of Algorithm \ref{algorithm:IR-push-pull} by introducing the errors metrics $\left\|\bar x_{k+1}-x^*_{\lambda_k}\right\|_2$,  $\|\mathbf{x}_{k+1}-\mathbf{1}\bar x_{k+1}\|_{\mathbf{R}}$, $\left\|\mathbf{y}_{k+1} - \nu \bar y_{k+1}\right\|_{\mathbf{C}}$. Of these, the first term relates the averaged iterate with the Tikhonov trajectory, the second term measures the consensus violation for the decision matrix, and the third term measures the consensus violation for the matrix of the regularized gradients.  For $k\geq 1$, let us define $\Delta_k$ as  $\Delta_k \triangleq \left[\|\bar x_k-x^*_{\lambda_{k-1}}\|_2, \left\|\mathbf{x}_k-\mathbf{1}\bar x_k\right\|_{\mathbf{R}}, \left\|\mathbf{y}_k-\nu \bar y_k\right\|_{\mathbf{C}}\right]^T$. 
\begin{proposition}\label{prop:main_ineq_1} 
Consider Algorithm \ref{algorithm:IR-push-pull} under Assumptions \ref{assum:problem}, \ref{assum:RC}, and \ref{assum:update_rules}. Let $\alpha_k $ and $\hat\gamma_k$ be given by Assumption \ref{assum:update_rules}, and $c_0\triangleq \delta_{\mathbf{C},2}\left\|\mathbf{I}-\tfrac{1}{m}\nu\mathbf{1}^T\right\|_\mathbf{C}$. Then, there exist scalars $M>0$, $B_\mathbf{g}>0$, and an integer $K$ such that for any $k\geq K$, we have $\Delta_{k+1} \leq H_k\Delta_k+h_k$ where $H_k =[H_{ij,k}]_{3\times 3}$ and $h_k =[h_{i,k}]_{3\times 1}$ are given as follows:
\begingroup
\allowdisplaybreaks
\begin{align*}
&H_{11,k} :=\textstyle{ 1-\mu_f\alpha_k  \lambda_k}, \ 
H_{12,k} := \textstyle{\tfrac{\alpha_kL_k}{\sqrt{m}}}, \ 
H_{13,k} :=\textstyle{\tfrac{\hat\gamma_k\|u\|_2}{m}},\\
&H_{21,k} :=  \textstyle{\sigma_{\mathbf{R}}\hat\gamma_k L_k\|\nu\|_{\mathbf{R}}},\ 
H_{22,k} := \textstyle{\sigma_{\mathbf{R}}\left(1+\hat\gamma_k\|\nu\|_{\mathbf{R}}\tfrac{L_k}{\sqrt{m}}\right)} ,\\ 
&H_{23,k} :=  \textstyle{\sigma_{\mathbf{R}}\hat\gamma_k\delta_{\mathbf{R},\mathbf{C}}}, \ 
H_{33,k} :=\textstyle{\sigma_{\mathbf{C}} +c_0L_{k}\hat\gamma_k \|\mathbf{R}\|_2},\\
&H_{31,k} := \textstyle{c_0L_{k}\left(\hat\gamma_k\|\mathbf{R}\|_2 \|\nu\|_2L_k+2\sqrt{m}\Lambda_{k}\right)},\\
&H_{32,k} := \textstyle{c_0L_{k}\left(\left\| \mathbf{R}-\mathbf{I}\right\|_2 +\hat\gamma_k\|\mathbf{R}\| \|\nu\|_2 \tfrac{L_k}{\sqrt{m}}   + 2\Lambda_{k}\right)},\\
&h_{1,k}:=\textstyle{\tfrac{M\Lambda_{k-1}}{\mu_f}}, \ 
h_{2,k}:=\textstyle{\tfrac{M\sigma_{\mathbf{R}}\hat\gamma_k L_k\|\nu\|_{\mathbf{R}}}{\mu_f}\Lambda_{k-1}},\\
 & h_{3,k}:= \textstyle{c_0L_{k}\left(\hat\gamma_k\|\mathbf{R}\|_2 \|\nu\|_2 L_k+\sqrt{m}\Lambda_{k}+\tfrac{\mu_f c_0B_{\mathbf{g}}}{M}\right) \tfrac{M\Lambda_{k-1}}{\mu_f}}.
\end{align*}
\endgroup
\end{proposition}

\begin{proof} 
\noindent First, we show $\Delta_{1,k+1}\leq \textstyle\sum_{j=1}^3 H_{1j,k}\Delta_{j,k} +h_{1,k} $. From \eqref{alg:IRPP_compact1} and Def. \ref{eqn:defs}, we obtain: $$\bar x_{k+1} =  u^T\mathbf{R}\left(\mathbf{x}_k -\boldsymbol{\gamma}_k\mathbf{y}_k\right)/m = \bar x_k - u^T\boldsymbol{\gamma}_k\mathbf{y}_k/m.$$ Thus, we have:
\begin{align*}
\bar x_{k+1} &=\bar x_k - u^T\boldsymbol{\gamma}_k\left(\mathbf{y}_k-\nu \bar y_k+\nu\bar y_k\right)/m 
\\
& = \bar x_k -\alpha_k  \bar g_k -\alpha_k  \left(\bar y_k -\bar g_k\right) - u^T\boldsymbol{\gamma}_k\left(\mathbf{y}_k-\nu\bar y_k\right)/m.
\end{align*}
From Assumption \ref{assum:update_rules}, $\alpha_k   <\tfrac{2}{L_0}\leq \tfrac{2}{L_k}$ for all $k\geq K$ for some $K$. From Lemma \ref{lemma:grad_track_props}(iii), $\mathscr{G}_k(x)$ is $(\mu_f\lambda_k)$-strongly convex and $L_k$-smooth. Invoking Lemma 10 in \cite{QuLi18}, we have for $\alpha_k   <\tfrac{2}{L_k}$ that 
$\left\|\bar x_k -\alpha_k  \bar g_k -x^*_{\lambda_k}\right\|_2 \leq \max\left(|1-\mu_f\lambda_k\alpha_k  |,|1-L_k\alpha_k  |\right)\left\|\bar x_k-x^*_{\lambda_k}\right\|_2$.
Thus, since $\mu_f\lambda_k \leq L_k$, we obtain for $\alpha_k   \leq\tfrac{1}{L_k}$ that $\left\|\bar x_k -\alpha_k  \bar g_k -x^*_{\lambda_k}\right\| \leq \left(1-\mu_f\lambda_k\alpha_k  \right)\left\|\bar x_k-x^*_{\lambda_k}\right\|$.
Using the preceding two relations, we obtain:
\begin{align*}
&\left\|\bar x_{k+1}-x^*_{\lambda_k}\right\|_2 
 = \left\|\bar x_k -x^*_{\lambda_k}-\alpha_k  \bar g_k -\alpha_k  \left(\bar y_k -\bar g_k\right) \right. \\ &\left. -\tfrac{1}{m}u^T\boldsymbol{\gamma}_k\left(\mathbf{y}_k-\nu\bar y_k\right)\right\|_2 
 \leq \left(1-\mu_f\alpha_k  \lambda_k\right)\|\bar x_k-x^*_{\lambda_k}\|_2 \\&+\alpha_k  \|\bar y_k -\bar g_k\|_2 +\tfrac{1}{m}\left\|u^T\boldsymbol{\gamma}_k\left(\mathbf{y}_k-\nu\bar y_k\right)\right\|_2.
\end{align*}
Adding and subtracting $x^*_{\lambda_{k-1}}$ and using Lemmas \ref{lemma:IR-props} and \ref{lemma:grad_track_props}(iv):  
\begin{align*}
&\left\|\bar x_{k+1}-x^*_{\lambda_k}\right\|_2 
 \leq \left(1-\mu_f\alpha_k  \lambda_k\right)\left\|\bar x_k-x^*_{\lambda_{k-1}}\right\|_2 +\tfrac{M\Lambda_{k-1}}{\mu_f}\\ &+\tfrac{\alpha_k  L_k}{\sqrt{m}}\|\mathbf{x}_k-\mathbf{1}\bar x_k\|_2+\tfrac{\|u\|_2\|\boldsymbol{\gamma}_k\|_2}{m}\left\|\mathbf{y}_k-\nu\bar y_k\right\|_2.
\end{align*}
Then, the desired inequality is obtained by invoking Lemma \ref{lemma:norms_relations}(ii), Remark \ref{rem:norms}, and definition of $\hat\gamma_k$.

\noindent  Second, we show $\Delta_{2,k+1}\leq \textstyle\sum_{j=1}^3 H_{2j,k}\Delta_{j,k} +h_{2,k} $. From \eqref{alg:IRPP_compact1} and Def. \ref{eqn:defs} and that $\mathbf{R}\mathbf{1} = \mathbf{1}$, we have:
\begin{align*}
\mathbf{x}_{k+1} - \mathbf{1} \bar x_{k+1}& \textstyle{= \mathbf{R}\left(\mathbf{x}_k-\boldsymbol{\gamma}_k \mathbf{y}_k\right) - \mathbf{1}\bar x_k +\tfrac{1}{m}\mathbf{1}u^T\boldsymbol{\gamma}_k \mathbf{y}_k} \\&\textstyle{= \left(\mathbf{R} -\mathbf{1}u^T/m\right)\left(\left(\mathbf{x}_k-\mathbf{1} \bar x_k\right) -\boldsymbol{\gamma}_k  \mathbf{y}_k\right).}
\end{align*}
Applying Lemma \ref{lemma:norms_relations}, Remark \ref{rem:norms}, and Lemma \ref{lemma:grad_track_props}, we obtain: 
\begin{align*}
&\|\mathbf{x}_{k+1}-\mathbf{1}\bar x_{k+1}\|_{\mathbf{R}} \leq \sigma_{\mathbf{R}}\left\|\mathbf{x}_k-\mathbf{1}\bar x_k\right\|_{\mathbf{R}}+ \sigma_{\mathbf{R}}\|\boldsymbol{\gamma}_k \|_{\mathbf{R}} \|\mathbf{y}_k\|_{\mathbf{R}}\\
&\leq \sigma_{\mathbf{R}}\left\|\mathbf{x}_k-\mathbf{1}\bar x_k\right\|_{\mathbf{R}}+ \sigma_{\mathbf{R}}\|\boldsymbol{\gamma}_k \|_{2}\|\mathbf{y}_k-\nu \bar y_k\|_{\mathbf{R}}\\ 
&+ \sigma_{\mathbf{R}}\|\boldsymbol{\gamma}_k \|_{2} \|\nu\|_{\mathbf{R}}\| \bar y_k\|_{2}\\
&\leq \sigma_{\mathbf{R}}\left(1+ \hat\gamma_k\|\nu\|_{\mathbf{R}}L_k/\sqrt{m}\right) \left\|\mathbf{x}_k-\mathbf{1}\bar x_k\right\|_{\mathbf{R}}\\
&+ \sigma_{\mathbf{R}}\hat\gamma_k\delta_{\mathbf{R},\mathbf{C}} \|\mathbf{y}_k-\nu \bar y_k\|_{\mathbf{C}}+ \sigma_{\mathbf{R}}\hat\gamma_k L_k\|\nu\|_{\mathbf{R}}\left\|\bar x_k-x^*_{\lambda_k}\right\|_2.
\end{align*}
Adding and subtracting $x^*_{\lambda_{k-1}}$ and using Lemma \ref{lemma:IR-props}, we obtain the desired inequality.

\noindent Third, we show $\Delta_{3,k+1}\leq \textstyle\sum_{j=1}^3 H_{3j,k}\Delta_{j,k} +h_{3,k} $.  From \eqref{alg:IRPP_compact2} and the definition of $\mathbf{G}_k(\mathbf{x})$ in Def. \ref{eqn:defs}, we obtain $\mathbf{y}_{k+1} = \mathbf{C}\mathbf{y}_k +\mathbf{G}_{k+1}\left(\mathbf{x}_{k+1}\right)-\mathbf{G}_{k}\left(\mathbf{x}_{k}\right).$ Multiplying both sides of the preceding relation by $\tfrac{1}{m}\mathbf{1}^T$ and using the definition of $\bar y_k$ in Def. \ref{eqn:defs}, we obtain that $\bar y_{k+1} = \bar y_{k}+\tfrac{1}{m}\mathbf{1}^T\mathbf{G}_{k+1}\left(\mathbf{x}_{k+1}\right)-\tfrac{1}{m}\mathbf{1}^T\mathbf{G}_{k}\left(\mathbf{x}_{k}\right)$. From the last two relations, we have:
\begin{align*}
&\mathbf{y}_{k+1} - \nu \bar y_{k+1} = \left(\mathbf{C}-\nu\mathbf{1}^T/m\right)\left(\mathbf{y}_k-\nu \bar y_k\right) \\ &+ \left(\mathbf{I}-\nu\mathbf{1}^T/m\right)\left(\mathbf{G}_{k+1}\left(\mathbf{x}_{k+1}\right)-\mathbf{G}_{k}\left(\mathbf{x}_{k}\right)\right).
\end{align*}
Invoking Lemma \ref{lemma:norms_relations}, $\mathbf{G}_k(\mathbf{x})$ in Def. \ref{eqn:defs} and $c_0$, and we obtain:
\begin{align}
&\left\|\mathbf{y}_{k+1} - \nu \bar y_{k+1}\right\|_{\mathbf{C}} \leq \sigma_{\mathbf{C}}\left\|\mathbf{y}_k-\nu \bar y_k\right\|_{\mathbf{C}} \notag\\ &+ c_0\left\|\mathbf{G}_{k+1}\left(\mathbf{x}_{k+1}\right)-\mathbf{G}_{k}\left(\mathbf{x}_{k}\right)\right\|_2\notag\\
&\leq \sigma_{\mathbf{C}}\left\|\mathbf{y}_k-\nu \bar y_k\right\|_{\mathbf{C}} + c_0\left\|\lambda_{k+1}\nabla \mathbf{f}(\mathbf{x}_{k})-\lambda_{k}\nabla \mathbf{f}(\mathbf{x}_{k})\right\|_2\notag\\ 
&+c_0\left\|\mathbf{G}_{k+1}\left(\mathbf{x}_{k+1}\right) -\nabla \mathbf{g}(\mathbf{x}_k)  -\lambda_{k+1}\nabla \mathbf{f}(\mathbf{x}_k)\right\|_2\notag\\
&\leq \sigma_{\mathbf{C}}\left\|\mathbf{y}_k-\nu \bar y_k\right\|_{\mathbf{C}}+c_0\left|1- \lambda_{k+1}/\lambda_{k}\right|\left\|\lambda_k\nabla \mathbf{f}(\mathbf{x}_{k})\right\|_2\notag\\ 
&+ c_0L_{k}\left\|\mathbf{x}_{k+1}-\mathbf{x}_k\right\|_2.\label{ineq:main_lemma_conv_part_c}
\end{align}
 From Lemma \ref{lemma:IR-props}, there exists a scalar $B_{\mathbf{g}}<\infty$ such that $L_g\|\mathbf{1}x^*_{\lambda_k}-\mathbf{1}x^*\|_2\leq B_{\mathbf{g}}$. From $\nabla g(x^*)=0$:
\begin{align*}
& \|\lambda_k\nabla \mathbf{f}(\mathbf{x}_{k})\|_2 \leq \|\nabla \mathbf{g}(\mathbf{x}_{k})+\lambda_k\nabla \mathbf{f}(\mathbf{x}_{k})\|_2\\ 
& +\|\nabla \mathbf{g}(\mathbf{x}_{k})-\nabla \mathbf{g}(\mathbf{1}x^*)\|_2\\
& \leq \|\nabla \mathbf{g}(\mathbf{x}_{k})+\lambda_k\nabla \mathbf{f}(\mathbf{x}_{k}) - \nabla \mathbf{g}(\mathbf{1}x^*_{\lambda_k})-\lambda_k\nabla \mathbf{f}(\mathbf{1}x^*_{\lambda_k}) \|_2 \\ 
&+ L_g\|\mathbf{x}_{k}-\mathbf{1}x^*\|_2 \\
&\leq (L_k+L_g)\|\mathbf{x}_{k}-\mathbf{1}x^*_{\lambda_k}\|_2+ L_g\|\mathbf{1}x^*_{\lambda_k}-\mathbf{1}x^*\|_2 \\ 
&\leq 2L_k\left(\|\mathbf{x}_{k}-\mathbf{1}\bar x_{k}\|_2+\|\mathbf{1}\bar x_{k}-\mathbf{1}x^*_{\lambda_k}\|_2\right)+ B_{\mathbf{g}} \\
&\leq 2L_k\|\mathbf{x}_{k}-\mathbf{1}\bar x_{k}\|_2+2\sqrt{m}L_k\|\bar x_{k}-x^*_{\lambda_k}\|_2+ B_{\mathbf{g}}.
\end{align*}
From row-stochasticity of $\mathbf{R}$, we have $\left(\mathbf{R}-\mathbf{I}\right)\mathbf{1}\bar x_k =0$. Thus, from Lemma \ref{lemma:grad_track_props} we have:
\begin{align*}
&\left\|\mathbf{x}_{k+1}-\mathbf{x}_k\right\|_2 =  \left\|\mathbf{R}\left(\mathbf{x}_k-\boldsymbol{\gamma}_k \mathbf{y}_k\right)- \mathbf{x}_k\right\|_2  \\ &=  \left\| \left(\mathbf{R}-\mathbf{I}\right)\left(\mathbf{x}_k-\mathbf{1}\bar x_k \right)-\mathbf{R}\boldsymbol{\gamma}_k \mathbf{y}_k\right\|_2\\
& \leq  \left\| \mathbf{R}-\mathbf{I}\right\|_2\left\|\mathbf{x}_k-\mathbf{1}\bar x_k \right\|_2+\|\mathbf{R}\|_2\|\boldsymbol{\gamma}_k \|_2(\|\mathbf{y}_k-\nu\bar y_k\|_2\\ &+\|\nu\|_2\|\bar y_k-\bar g_k\|_2+\|\nu\|_2\|\bar g_k\|_2)\leq   \left\| \mathbf{R}-\mathbf{I}\right\|_2\left\|\mathbf{x}_k-\mathbf{1}\bar x_k \right\|_2\\
&+\hat\gamma_k\|\mathbf{R}\|_2\left(\|\mathbf{y}_k-\nu\bar y_k\|_2\right. \\ & \left.+L_k \|\nu\|_2\left(\|\mathbf{x}_k-\mathbf{1}\bar x_k\|_2/\sqrt{m}+\|\bar x_k-x^*_{\lambda_k}\|_2\right)\right).
\end{align*}
It suffices to find a recursive bound for the term $\left\|\mathbf{x}_{k+1}-\mathbf{x}_k\right\|_2$. From Lemma \ref{lemma:IR-props}, we can write: 
\begin{align*}
\|\bar x_k-x^*_{\lambda_k}\|_2&\leq  \|\bar x_k-x^*_{\lambda_{k-1}}\|_2+\|x^*_{\lambda_{k-1}}-x^*_{\lambda_k}\|_2 \\ &\leq \|\bar x_k-x^*_{\lambda_{k-1}}\|_2+  M/\mu_f\Lambda_{k-1}.
\end{align*}
From \eqref{ineq:main_lemma_conv_part_c}, the preceding three relations, we can obtain the desired inequality.
\end{proof}
Next, we derive a unifying recursive bound for the three error bounds introduced earlier to be used in deriving the rates. 
\begin{proposition}\label{prop:recursive_bound_for_rate} Consider Algorithm \ref{algorithm:IR-push-pull}. Let Assumptions \ref{assum:problem}, \ref{assum:RC}, and \ref{assum:update_rules} hold. Then, there exists an integer $\mathscr{K}\geq 1$ such that for any $k\geq \mathscr{K}$, the following holds:

\noindent (a) $\|\Delta_{k+1}\|_2 \leq (1-0.5\mu_f\alpha_k\lambda_k)\|\Delta_k\|_2 +\Theta\Lambda_{k-1}$, where 
\begin{align*}
\Theta &\triangleq \max\left\{1,\sigma_{\mathbf{R}}\hat\gamma_0 L_0\|\nu\|_{\mathbf{R}},c_0L_{0}\left(\hat\gamma_0\|\mathbf{R}\|_2 \|\nu\|_2 L_0\right. \right.\\ &\left.\left.+\sqrt{m}\Lambda_{0}+\mu_f c_0B_{\mathbf{g}}/M\right)\right\}\sqrt{3}M/\mu_f.
\end{align*}

\noindent (b) There exists a scalar $\mathscr{B}>0$ such that $\|\Delta_{k}\|_2 \leq \tfrac{\mathscr{B}}{k^{1-a-b}}$.

\end{proposition}
\begin{proof}
\noindent \textbf{(a)} In the first step, we consider Proposition \ref{prop:main_ineq_1}. 
Let us define the sequence $\{\rho_k\}$ as $\rho_k \triangleq 1-0.5\mu_f\alpha_k  \lambda_k$ for $k\geq 0$. Next, we the utilize our assumptions to find suitable upper bounds for some of the above terms. We define $\hat H_k =[\hat H_{ij,k}]_{3\times 3}$ and $\hat h_k =[\hat h_{i,k}]_{3\times 1}$ as follows:
\begin{align*}
&\hat H_{11,k} := H_{11,k} , \quad 
\hat H_{12,k} := \tfrac{\alpha_k   L_0}{\sqrt{m}}, \  
\hat H_{13,k} :=H_{13,k} ,\\
&\hat H_{21,k} :=  \sigma_{\mathbf{R}}\hat\gamma_k L_0\|\nu\|_{\mathbf{R}},\  
\hat H_{22,k} := \rho_k-\tfrac{1-\sigma_{\mathbf{R}}}{2} ,\\  
&\hat H_{23,k} :=  H_{23,k}, \\
&\hat H_{31,k} := c_0L_{0}\left(\hat\gamma_k\|\mathbf{R}\|_2 \|\nu\|_2L_0+2\sqrt{m}\Lambda_{k}\right),\\
&\hat H_{32,k} := c_0L_{0}\left(\left\| \mathbf{R}-\mathbf{I}\right\|_2 +\hat\gamma_k\|\mathbf{R}\| \|\nu\|_2 \tfrac{L_0}{\sqrt{m}}   + 2\Lambda_{0}\right),\\
&\hat H_{33,k} := \rho_k-\tfrac{1-\sigma_{\mathbf{C}}}{2} ,\\
&\hat h_{1,k}:=\tfrac{\Theta}{\sqrt{3}}\Lambda_{k-1}, \quad 
\hat h_{2,k}:=\tfrac{\Theta}{\sqrt{3}}\Lambda_{k-1},\quad 
\hat h_{3,k}:= \tfrac{\Theta}{\sqrt{3}}\Lambda_{k-1}.
\end{align*}
Note that we have: 
\begin{align*}
\hat H_{22,k}- H_{22,k} 
&=\tfrac{1-\sigma_{\mathbf{R}}}{2} -0.5\mu_f\alpha_k\lambda_k-\hat\gamma_k\|\nu\|_{\mathbf{R}}\tfrac{L_k}{\sqrt{m}}.
\end{align*} 
From Assumption \ref{assum:update_rules} and the definition of $\alpha_k$, we have $\hat \gamma_k\to 0$, $\alpha_k\to 0$, and $\lambda_k \to 0$. Thus, there exists an integer $K_\mathbf{R}\geq 1$ such that for all $k \geq K_{\mathbf{R}}$ we have $H_{22,k}\leq \hat H_{22,k}$. Similarly, there exists an integer $K_\mathbf{C}\geq 1$ such that for all $k \geq K_{\mathbf{C}}$ we have $H_{33,k}\leq \hat H_{33,k}$. Thus, by taking to account that $\lambda_k$ and $\Lambda_k$ are nonincreasing sequences and invoking the definition of $\Theta$, we have $H_k\leq \hat H_k$ and $h_{k}\leq \hat h_{k}$. This implies that for all $k\geq max\{K, K_{\mathbf{R}}, K_{\mathbf{C}}\}$, we have $\Delta_{k+1} \leq \hat H_k\Delta_k+\hat h_k$. Consequently, we obtain:
\begin{align}\label{eqn:rec_Delta_proof}
\textstyle{\left\|\Delta_{k+1}\right\|_2 \leq \rho\left(\hat H_k\right)\left\|\Delta_k\right\|_2+\Theta \Lambda_{k-1},}
\end{align}
where $\rho\left(\hat H_k\right)$ denotes the spectral norm of $\hat H_k$. Next, we show that for a sufficiently large $k$, we have $\rho\left(\hat H_k\right) \leq \rho_k $. To show this relation, employing Lemma 5 in \cite{pushi2_2020}, it suffices to show that $0\leq \hat H_{ii,k} <\rho_k$ for $i \in \{1,2,3\}$ and $\text{det}\left(\rho_k\mathbf{I}-\hat H_k\right)>0$. Among these, it can be easily seen that $H_{ii,k} <\rho_k$ holds for all $i \in \{1,2,3\}$. Since $\alpha_k \to 0$ and $\lambda_k\to 0$, there exists an integer $K_1$ such that $\hat H_{11,k} = 1-\mu_f\alpha_k  \lambda_k>0$. Similarly, from $\sigma_{\mathbf{R}}<1$ and $\sigma_{\mathbf{C}}<1$, there exists integers $K_2$ and $K_3$ such that $\hat H_{22,k}>0$ and $\hat H_{33,k}>0$, respectively. Next, we show $\text{det}\left(\rho_k\mathbf{I}-\hat H_k\right)>0$. 
\begin{align*}
&\textstyle{\text{det}\left(\rho_k\mathbf{I}-\hat H_k\right) }
 \textstyle{= \left(0.5\mu_f\alpha_k \lambda_k\right)\left( \tfrac{1-\sigma_{\mathbf{R}}}{2}\right)\left( \tfrac{1-\sigma_{\mathbf{C}}}{2}\right)}\\
&\textstyle{- \left(0.5\mu_f\alpha_k \lambda_k\right)\left(\sigma_{\mathbf{R}}\hat\gamma_k\delta_{\mathbf{R},\mathbf{C}}\right)c_0L_{0}\left(\left\| \mathbf{R}-\mathbf{I}\right\|_2+ 2\Lambda_{0}\right.} \\ 
&\textstyle{\left. +\hat\gamma_k\|\mathbf{R}\| \|\nu\|_2 \tfrac{L_0}{\sqrt{m}}   \right)-\left(\tfrac{1-\sigma_{\mathbf{C}}}{2}\right)\left( \tfrac{\alpha_k   L_0}{\sqrt{m}}\right)\left( \sigma_{\mathbf{R}}\hat\gamma_k L_0\|\nu\|_{\mathbf{R}}\right)}\\
&\textstyle{ -\left(\tfrac{\alpha_k   L_0}{\sqrt{m}}\right)\left(\sigma_{\mathbf{R}}\hat\gamma_k\delta_{\mathbf{R},\mathbf{C}}\right)\left(c_0L_{0}\left(\hat\gamma_k\|\mathbf{R}\|_2 \|\nu\|_2L_0+2\sqrt{m}\Lambda_{k}\right)\right)}\\
& \textstyle{-\left(\tfrac{\hat\gamma_k\|u\|_2}{m}\right)\left(\sigma_{\mathbf{R}}\hat\gamma_k L_0\|\nu\|_{\mathbf{R}}\right)\left(c_0L_{0}\left(\left\| \mathbf{R}-\mathbf{I}\right\|_2  \right. \right.} \\
&\textstyle{ \left.\left. +\hat\gamma_k\|\mathbf{R}\| \|\nu\|_2 \tfrac{L_0}{\sqrt{m}}   + 2\Lambda_{0}\right)\right)}\\
&\textstyle{ - \left(\tfrac{1-\sigma_{\mathbf{R}}}{2}\right)\left(c_0L_{0}\left(\hat\gamma_k\|\mathbf{R}\|_2 \|\nu\|_2L_0+2\sqrt{m}\Lambda_{k}\right)\right)\left(\tfrac{\hat\gamma_k\|u\|_2}{m}\right).}
\end{align*}
Next, we find lower and upper bounds on $\alpha_k$ in terms of $\hat \gamma_k$. Assumption \ref{assum:update_rules} provides $\theta\hat \gamma_k$ as a lower bound for $\alpha_k$. To find an upper bound, from Lemma 1 in \cite{pushi1_2020}, we have that the eigenvector $u$ is nonzero only on the entries $i \in \mathcal{R}_{\mathbf{R}}$. Similarly, the eigenvector $\nu$ is nonzero only on the entries $i \in \mathcal{R}_{\mathbf{C^T}}$. Also, we have $u^T\nu>0$. Thus, 
we can write:
\begin{align*}
\tfrac{\alpha_k}{\hat \gamma_k}&\textstyle{=\tfrac{1}{m}u^T\tfrac{\boldsymbol{\gamma}_k}{\hat \gamma_k}\nu=\tfrac{1}{m}\textstyle\sum\nolimits_{i \in\mathcal{R}_{\mathbf{R}\cap\mathbf{C}^T} }u_iv_i\tfrac{\gamma_{i,k}}{\hat \gamma_k} }\\ 
&\textstyle{\leq \tfrac{1}{m}\textstyle\sum\nolimits_{i \in\mathcal{R}_{\mathbf{R}\cap\mathbf{C}^T} }u_iv_i\tfrac{\gamma_{i,k}}{\hat \gamma_k} =\tfrac{1}{m}u^T\nu>0.}
\end{align*}
Let us define $\bar \theta = \tfrac{1}{m}u^T\nu$. Thus, we have $\theta\hat \gamma_k \leq \alpha_k \leq \bar \theta \hat \gamma_k$ for all $k\geq 0$. Using these bounds and rearranging the terms:
\begin{align*}
\text{det}\left(\rho_k\mathbf{I}-\hat H_k\right) \geq  -c_1\hat\gamma_k^3-c_2\hat\gamma_k^2+c_3\hat\gamma_k\lambda_k-c_4\hat\gamma_k\Lambda_k,
\end{align*}
where the scalars $c_1$, $c_2$, $c_3$ are defined as below:
\begin{align*}
c_1 &\textstyle{\triangleq \left(0.5\mu_f\bar \theta \lambda_0\right)\left(\sigma_{\mathbf{R}}\delta_{\mathbf{R},\mathbf{C}}\right)c_0L_{0}\left(\|\mathbf{R}\| \|\nu\|_2 \tfrac{L_0}{\sqrt{m}}\right)}\\
&\textstyle{+\left(\tfrac{\bar \theta   L_0}{\sqrt{m}}\right)\left(\sigma_{\mathbf{R}}\delta_{\mathbf{R},\mathbf{C}}\right)\left(c_0L_{0}\left(\|\mathbf{R}\|_2 \|\nu\|_2L_0\right)\right)}\\
& \textstyle{+\left(\tfrac{\|u\|_2}{m}\right)\left(\sigma_{\mathbf{R}} L_0\|\nu\|_{\mathbf{R}}\right)\left(c_0L_{0}\|\mathbf{R}\| \|\nu\|_2 \tfrac{L_0}{\sqrt{m}}   \right)}\\
c_2 &\textstyle{\triangleq \left(0.5\mu_f\bar \theta \lambda_0\right)\left(\sigma_{\mathbf{R}}\delta_{\mathbf{R},\mathbf{C}}\right)c_0L_{0}\left(\left\| \mathbf{R}-\mathbf{I}\right\|_2 + 2\Lambda_{0}\right)}\\
&\textstyle{+\left(\tfrac{1-\sigma_{\mathbf{C}}}{2}\right)\left( \tfrac{\bar \theta  L_0}{\sqrt{m}}\right)\left( \sigma_{\mathbf{R}} L_0\|\nu\|_{\mathbf{R}}\right)}\\
&\textstyle{+\left(\tfrac{\bar \theta   L_0}{\sqrt{m}}\right)\left(\sigma_{\mathbf{R}}\delta_{\mathbf{R},\mathbf{C}}\right)\left(c_0L_{0}2\sqrt{m}\Lambda_{0}\right)}\\
&\textstyle{+\left(\tfrac{\|u\|_2}{m}\right)\left(\sigma_{\mathbf{R}} L_0\|\nu\|_{\mathbf{R}}\right)\left(c_0L_{0}\left(\left\| \mathbf{R}-\mathbf{I}\right\|_2  + 2\Lambda_{0}\right)\right)}\\
&\textstyle{+\left(\tfrac{1-\sigma_{\mathbf{R}}}{2}\right)\left(c_0L_{0}\left(\mathbf{R}\|_2 \|\nu\|_2L_0\right)\right)\left(\tfrac{\|u\|_2}{m}\right)}\\
c_3 &\textstyle{\triangleq (0.5)^3\mu_f\theta\left(1-\sigma_{\mathbf{R}}\right)\left( 1-\sigma_{\mathbf{C}}\right)}\\
c_4 &\textstyle{\triangleq\left(\tfrac{1-\sigma_{\mathbf{R}}}{2}\right)\left(c_0L_{0}\sqrt{m}\right)\left(\tfrac{\|u\|_2}{m}\right).}
\end{align*}
It suffices to show that $-c_1\hat\gamma_k^3-c_2\hat\gamma_k^2+c_3\hat\gamma_k\lambda_k-c_4\hat\gamma_k\Lambda_k>0$ for any sufficiently large $k$. From Lemma \ref{lem:update_rules_props}, we have $\tfrac{\Lambda_k}{\lambda_k}\to 0$. Thus, there exists an integer $K_4\geq 0$ such that for any $k \geq K_4$ we have that $c_4\Lambda_k \leq 0.5c_3\lambda_k$.  As such, it suffices to show that $ c_1\hat\gamma_k^2+c_2\hat\gamma_k<0.5c_3\lambda_k $. From Lemma \ref{lem:update_rules_props}, since $\hat \gamma_k \to 0$ and $\tfrac{\hat \gamma}{\lambda_k}\to 0$, there exists an integer $K_5\geq 0$ such that $ c_1\hat\gamma_k^2+c_2\hat\gamma_k<0.5c_3\lambda_k $ for any $k\geq K_5$. We conclude that for $\mathscr{K} \triangleq  \max\{K,K_1,K_2,K_3,,K_4,K_5,K_{\mathbf{R}},K_{\mathbf{C}}\}$, we have $\text{det}\left(\rho_k\mathbf{I}-\hat H_k\right)>0$ for any $k \geq \mathscr{K}$.  Therefore, we have $\rho\left(\hat H_k\right) \leq  1-0.5\mu_f\alpha_k  \lambda_k$ for all $k \geq\mathscr{K}$. The desired inequality is obtained from this and the relation \eqref{eqn:rec_Delta_proof}.

\noindent \textbf{(b)} From Lemma \ref{lem:update_rules_props}, we have that $\Lambda_{k-1} \leq \tfrac{1}{k+1}$. From part (a) and Assumption \ref{assum:update_rules}, we obtain for all $k\geq \mathscr{K}$:
\begin{align}\label{eqn:rec_Delta_proof_v2}
\textstyle{\|\Delta_{k+1}\|_2 \leq (1-0.5\mu_f\alpha_k\hat \gamma_k\theta)\|\Delta_k\|_2 +\tfrac{\Theta}{k+1}.}
\end{align}
We use induction to show that the desired relation holds for: 
$$\textstyle{\mathscr{B}\triangleq \max\left\{(\mathscr{K}+1)^{1-a-b}\|\Delta_{\mathscr{K}}\|_2,\tfrac{4\Theta}{\mu_f\lambda_0\hat \gamma_0 \theta}\right\}.}$$ 
First, note that the inequality holds for $k:=\mathscr{K}$. Let us assume that $\|\Delta_{k}\|_2 \leq \tfrac{\mathscr{B}}{k^{1-a-b}}$ for some $k \geq \mathscr{K}$. We show that this relation also holds for $k+1$. Consider Lemma \ref{lem:update_rules_props}. Let us choose $\tau :=\tfrac{\mu_f\theta}{4}$. Thus, from  Lemma \ref{lem:update_rules_props}, there exists a $K_\tau$ such that $\tfrac{(k+1)\hat\gamma_k\lambda_k}{k\hat\gamma_{k-1}\lambda_{k-1}}\leq 1+\tau\hat \gamma_k\lambda_k$ for all $k\geq K_{\tau}$. Thus:
\begin{align}\label{eqn:ineq1_for_induction}
\textstyle{\tfrac{k^{a+b}}{k}\leq \tfrac{(k+1)^{a+b}}{k+1}(1+0.25\mu_f\lambda_0\hat \gamma_0\theta).}
\end{align}
Let $K_6$ be an integer such that $0.5\mu_f\alpha_k\hat \gamma_k\theta<1$. Without loss of generality, let us assume $\mathscr{K} \geq \max\left\{K_\tau,K_6\right\}$. From \eqref{eqn:rec_Delta_proof_v2} and the induction hypothesis, we obtain:
\begin{align*}
\textstyle{\|\Delta_{k+1}\|_2 \leq (1-0.5\mu_f\alpha_k\hat \gamma_k\theta) \tfrac{\mathscr{B}}{k^{1-a-b}} +\tfrac{\Theta}{k+1}.}
\end{align*}
From the preceding relation and \eqref{eqn:ineq1_for_induction}, we obtain:
\begin{align*}
\textstyle{\|\Delta_{k+1}\|_2 \leq   \dfrac{\mathscr{B}(1-0.5\mu_f\alpha_k\hat \gamma_k\theta)(1+0.25\mu_f\alpha_k\hat \gamma_k\theta)}{(k+1)^{1-a-b}} +\tfrac{\Theta}{k+1}}.
\end{align*}
From the definition of $\mathscr{B}$ we have $\Theta \leq 0.25\mu_f \lambda_0 \hat \gamma_0\theta\mathscr{B}$. Therefore, we obtain:
\begin{align*}
\|\Delta_{k+1}\|_2 \leq  &\textstyle{\dfrac{\mathscr{B}\left(1-0.25\mu_f \lambda_0 \hat \gamma_0\theta -0.125(\mu_f \lambda_0 \hat \gamma_0\theta)^2\right)}{(k+1)^{1-a-b}} }\\ &{+ \tfrac{0.25\mu_f \lambda_0 \hat \gamma_0\theta\mathscr{B}}{(k+1)^{1-a-b}}}.
\end{align*}
This implies that $\|\Delta_{k+1}\|_2 \leq \tfrac{\mathscr{B}}{(k+1)^{1-a-b}}$. Thus, the induction statement holds for $k+1$ and hence, the proof is completed.
\end{proof}
Our first main result is provided below where we derive a family of convergence rates for the bilevel formulation \eqref{eqn:bilevel_problem}. 
\begin{theorem}[Rate statements for the bilevel model]\label{thm:rate_ana}
Consider problem \eqref{eqn:bilevel_problem} and Algorithm \ref{algorithm:IR-push-pull}. Let Assumptions \ref{assum:problem}, \ref{assum:RC}, and \ref{assum:update_rules} hold. Then, we have the following results: 

\noindent (a) We have $\lim_{k\to\infty}\bar x_{k}=x^*$. Also, the consensus violation of $\mathbf{x}_k$ and $\mathbf{y}_k$ characterized by $\|\mathbf{x}_{k+1}-\mathbf{1}\bar x_{k+1}\|_{\mathbf{R}}$ and $\left\|\mathbf{y}_{k+1} - \nu \bar y_{k+1}\right\|_{\mathbf{C}}$, respectively, are both bounded by $\mathcal{O}\left(1/k^{1-a-b}\right)$ for any sufficiently large $k$.

\noindent (b) We have $ f(\bar x_k)- f(x^*)\leq \tfrac{\mathscr{Q}_1\left(L_g+\lambda_0L_f\right)}{2}\tfrac{1}{k^{2-2a-3b}}$ for some $\mathscr{Q}_1>0$ and any sufficiently large $k$.

\noindent (c) $g(\bar x_k)- g(x^*)\leq \tfrac{\mathscr{Q}_2\left(L_g+\lambda_0L_f\right)}{2}\tfrac{1}{k^{2-2a-2b}}+\tfrac{\lambda_0\mathscr{Q}_3}{k^b}$ for  $\mathscr{Q}_2,\mathscr{Q}_3>0$ and any sufficiently large $k$.
\end{theorem}
\begin{proof}
\noindent \textbf{(a)} From Lemma \ref{lemma:IR-props}(a), we have that $\{x^*_{\lambda_k}\}$ converges to $x^*$. Moreover, from Proposition \ref{prop:recursive_bound_for_rate}(b), we have that $\|\bar x_{k}-x^*_{\lambda_{k-1}}\|_2$ converges to zero. Therefore, we have $\lim_{k \to \infty} \bar x_k = x^*$. To derive the bounds for $\left\|\mathbf{x}_k-\mathbf{1}\bar x_k\right\|_{\mathbf{R}}$ and $ \left\|\mathbf{y}_k-\nu \bar y_k\right\|_{\mathbf{C}}$, from the definition of $\Delta_k$ in Proposition \ref{prop:recursive_bound_for_rate}, we can write:
$\left\|\mathbf{x}_k-\mathbf{1}\bar x_k\right\|_{\mathbf{R}} \leq \|\Delta_k\|_2 =\mathcal{O}\left(k^{1-a-b}\right).$
Similarly, we obtain $\left\|\mathbf{y}_k-\nu \bar y_k\right\|_{\mathbf{C}}=\mathcal{O}\left(k^{1-a-b}\right)$.

\noindent \textbf{(b)} Consider the regularized function $g(x)+\lambda_k f(x)$. Note that it is $L_k$-smooth, where $L_k\triangleq L_g + \lambda_k L_f$. Since $x^*_{\lambda_k}$ is the minimizer of $g(x)+\lambda_k f(x)$, we have $x \in \mathbb{R}^n$:
\begin{align*}
g(x) +\lambda_k f(x) - g\left(x^*_{\lambda_k}\right) -\lambda_k f\left(x^*_{\lambda_k}\right) \leq \tfrac{L_k}{2}\left\|x- x^*_{\lambda_k}\right\|_2^2. 
\end{align*}
Also, we can write that $g\left(x^*_{\lambda_k}\right) +\lambda_k f\left(x^*_{\lambda_k}\right) \leq g\left(x^*\right) +\lambda_k f\left(x^*\right) $. Combining the preceding two relations and substituting $x$ by $\bar x_{k+1}$, we obtain $g\left(\bar x_{k+1}\right) - g\left(x^*\right) +\lambda_k\left(f\left(\bar x_{k+1}\right)- f\left(x^*\right)\right) \leq \tfrac{L_k}{2}\left\|\bar x_{k+1}- x^*_{\lambda_k}\right\|_2^2. $
Applying the bound from Proposition \ref{prop:recursive_bound_for_rate}(b), we obtain:
\begin{align}\label{eqn:thm1_part(b)_ineq1}
&g\left(\bar x_{k+1}\right) - g\left(x^*\right) +\lambda_k\left(f\left(\bar x_{k+1}\right)- f\left(x^*\right)\right)\notag \\ &\leq  \tfrac{L_k\mathscr{B}^2}{2(k+1)^{2-2a-2b}} \qquad \hbox{for all }k \geq \mathscr{K}. 
\end{align}
Note that from the definition of $x^*$ in Def. \ref{eqn:defs}, we have $g\left(\bar x_{k+1}\right) - g\left(x^*\right)\geq 0$. This implies that for all $k \geq \mathscr{K}$: 
\begin{align*}
f\left(\bar x_{k+1}\right)- f\left(x^*\right) 
\leq\left(\tfrac{L_0\mathscr{B}^2}{2\lambda_0}\right)\tfrac{1}{(k+1)^{2-2a-3b}}. 
\end{align*}
Therefore, the desired relation holds for $\mathscr{Q}_1\triangleq \tfrac{\mathscr{B}^2}{\lambda_0}$.

\noindent \textbf{(c)} From part (a), we know that $\{\bar x_{k}\}$ converges to $x^*$. This result and that $f$ is a continuous function imply that there exists a scalar $\mathscr{Q}_3>0$ such that $\left|f\left(\bar x_{k+1}\right)- f\left(x^*\right)\right| \leq \mathscr{Q}_3$. Thus, from the inequality \eqref{eqn:thm1_part(b)_ineq1} and the update rule for $\lambda_k$:
\begin{align*}
g\left(\bar x_{k+1}\right) - g\left(x^*\right) \leq \left(\tfrac{L_0\mathscr{B}^2}{2}\right)\tfrac{1}{(k+1)^{2-2a-2b}} +\tfrac{\mathscr{Q}_3\lambda_0}{(k+1)^b},
\end{align*}
for all $k \geq \mathscr{K}$. Therefore, the desired relation holds.
\end{proof}
In the following, we present the implications of the results of Theorem \ref{thm:rate_ana} in solving the constrained problem \eqref{eqn:lp_cnstr_problem}.
\begin{corollary}[Rates for the linearly constrained model]\label{cor:rate_lin_constr}
Consider problem \eqref{eqn:lp_cnstr_problem} and Algorithm \ref{algorithm:IR-push-pull} where $g_i(x)$ is defined by \eqref{eqn:gi}. Let the feasible set be nonempty and Assumption \ref{assum:problem}(a) and Assumption \ref{assum:RC} hold. Suppose Assumption \ref{assum:update_rules} holds with $a:=0.2$ and $b:=0.2-\epsilon/3$ where $\epsilon>0$ is a sufficiently small scalar. Then, we have $\lim_{k\to\infty}\bar x_{k}=x^*$ and for any sufficiently large $k$: 

\noindent (a) We have $\|\mathbf{x}_{k+1}-\mathbf{1}\bar x_{k+1}\|_{\mathbf{R}} =\mathcal{O}\left(1/k^{0.6+\epsilon/3}\right)$, and $\left\|\mathbf{y}_{k+1} - \nu \bar y_{k+1}\right\|_{\mathbf{C}}=\mathcal{O}\left(1/k^{0.6+\epsilon/3}\right)$.

\noindent (b) We have $f(\bar x_k)- f(x^*)=\mathcal{O}\left(1/k^{1-\epsilon}\right)$.

\noindent (c) We have $\|\mathbf{A}\bar x_k-\mathbf{b}\|_2^2 = \mathcal{O}\left(1/k^{0.2-\epsilon/3}\right)$ where $\mathbf{A} \triangleq \left[A_1^T,\ldots,A_m^T\right]^T $ and $\mathbf{b} \triangleq \left[b_1^T,\ldots,b_m^T\right]$. 
\end{corollary}
\begin{proof}
First, we show that problem \eqref{eqn:lp_cnstr_problem} is equivalent to problem \eqref{eqn:bilevel_problem}. Let $X_1$ and $X_2$ denote the feasible set of problem \eqref{eqn:bilevel_problem} and \eqref{eqn:lp_cnstr_problem}, respectively. Suppose $\hat x \in X_1$ is an arbitrary vector. Thus, we have $\hat x \in \argmin_{x \in \mathbb{R}^n}\tfrac{1}{2}\textstyle\sum_{i=1}^m \|A_ix-b_i\|_2^2+\tfrac{1}{2m}\textstyle\sum_{j\in \mathcal{J}}\max\{0,-x_j\}^2$. From the assumption $X_2\neq \emptyset$, there exists a point $\bar x$ satisfying $\mathbf{A}\bar x =\mathbf{b}$ and $\bar x_j\geq 0$ for all $j \in \mathcal{J}$. This implies that the minimum of the function $\tfrac{1}{2}\textstyle\sum_{i=1}^m \|A_ix-b_i\|_2^2+\tfrac{1}{2m}\textstyle\sum_{j\in \mathcal{J}}\max\{0,-x_j\}^2$ is zero. Therefore, $\hat x$ must satisfy $\mathbf{A} x =\mathbf{b}$ and $ x_j\geq 0$ for all $j \in \mathcal{J}$, implying that $\hat x \in X_2$. Next, suppose  $\hat x \in X_2$ is an arbitrary vector. Thus, we have $\tfrac{1}{2}\textstyle\sum_{i=1}^m \|A_i\hat x-b_i\|_2^2+\tfrac{1}{2m}\textstyle\sum_{j\in \mathcal{J}}\max\{0,-\hat x_j\}^2=0$ implying that $\hat x$ is a minimizer of the $\tfrac{1}{2}\textstyle\sum_{i=1}^m \|A_ix-b_i\|_2^2+\tfrac{1}{2m}\textstyle\sum_{j\in \mathcal{J}}\max\{0,-x_j\}^2$. Therefore, we have $\hat x \in X_1$. We conclude that $X_1 = X_2$ and thus problems \eqref{eqn:bilevel_problem} and \eqref{eqn:lp_cnstr_problem} are equivalent. Next, we show that Assumption \ref{assum:problem}(b) is satisfied. From the definition of function $g_i$ by \eqref{eqn:gi}, it is not hard to show that $\nabla  g_i(x)$ indeed exists and $\nabla g_i(x)=A_i^T\left(A_ix-b_i\right) - \tfrac{1}{m}\textstyle\sum_{j \in \mathcal{J}}\max\{0,-x_j\}\mathbf{e}_j$. 
Note that the mapping $A_i^T\left(A_ix-b_i\right)$ is Lipschitz with parameter $\rho\left(A_i^TA_i\right)$ denoting the spectral norm of $A_i^TA_i$. Also, it can be shown that the mapping $- \tfrac{1}{m}\textstyle\sum_{j \in \mathcal{J}}\max\{0,-x_j\}\mathbf{e}_j$ is Lipschitz with parameter $\tfrac{1}{\sqrt{m}}$ (proof omitted). Thus, we conclude that Assumption \ref{assum:problem}(b) is met for $L_g \triangleq \max_{i \in [m]}\rho\left(A_i^TA_i\right)+\tfrac{1}{\sqrt{m}}$. Therefore, all conditions of Theorem \ref{thm:rate_ana} hold. To obtain the rate results in part (a), (b), (c), it suffices to substitute $a$ by $0.2$ and $b$ by $0.2 -\tfrac{\epsilon}{3}$ in the corresponding parts in Theorem \ref{thm:rate_ana}.
\end{proof}
Lastly, we present the implications of the results of Theorem \ref{thm:rate_ana} in addressing the absence of strong convexity.
\begin{corollary}[Rates for problem \eqref{eqn:l2_uncnstr_problem}]\label{cor:rate_unconstr_nonstrongly}
Consider problem \eqref{eqn:l2_uncnstr_problem} and Algorithm \ref{algorithm:IR-push-pull} where we set $f_i(x) := \|x\|_2^2/m$. Let Assumption \ref{assum:problem}(b), \ref{assum:problem}(c) and Assumption \ref{assum:RC} hold. Suppose Assumption \ref{assum:update_rules} holds with $a:=0.4$ and $b:=0.4-\epsilon$ where $\epsilon>0$ is a sufficiently small scalar. Let $x^*_{\ell_2}$ denote the least $\ell_2$-norm optimal solution of problem \eqref{eqn:l2_uncnstr_problem}. Then, for any sufficiently large $k$: 

\noindent (a) We have $\|\mathbf{x}_{k+1}-\mathbf{1}\bar x_{k+1}\|_{\mathbf{R}} =\mathcal{O}\left(1/k^{0.2+\epsilon}\right)$ and $\left\|\mathbf{y}_{k+1} - \nu \bar y_{k+1}\right\|_{\mathbf{C}}=\mathcal{O}\left(1/k^{0.2+\epsilon}\right)$.

\noindent (b) We have $g(\bar x_k)-g\left(x^*_{\ell_2}\right)=\mathcal{O}\left(1/k^{0.4-\epsilon}\right)$ and that $\|\bar x_{k} - x^*_{\ell_2}\|_2^2  = \mathcal{O}\left(1/k^{3\epsilon}\right)$.
\end{corollary}
\begin{proof}
Note that problem \eqref{eqn:l2_uncnstr_problem} is equivalent to problem \eqref{eqn:bilevel_problem} where $f_i(x) := \|x\|_2^2/m$. This implies that Assumption \ref{assum:problem}(a) holds with $\mu_f = L_f =\tfrac{2}{m}$.  Therefore, all conditions of Theorem \ref{thm:rate_ana} hold. To obtain the rate results in part (a) and (b), it suffices to substitute $a$ by $0.4$ and $b$ by $0.4 -\epsilon$ in the rate results in Theorem \ref{thm:rate_ana}.
\end{proof}
\section{Numerical results}
\noindent \textbf{(1) Distributed sensor network problems:} We first compare Algorithm \ref{algorithm:IR-push-pull} with the Push-Pull algorithm \cite{pushi1_2020} in a sensor network example. We consider the unconstrained ill-posed problem $\min_{x \in \mathbb{R}^n}\textstyle\sum_{k=1}^m\|z_i-H_ix\|^2_2$, where $H_i \in \mathbb{R}^{d\times n}$ and $z_i \in \mathbb{R}^d$ denote the measurement matrix and the noisy observation of the $i^{\text{th}}$ sensor. Due to the challenges raised by ill-conditioning and also the lack of convergence and rate guarantees, Push-Pull algorithm needs to be applied to a regularized variant of the problem. To this end, in the implementation of the Push-Pull scheme, we use an $\ell_2$ regularizer with a parameter $0.1$. Accordingly, in Algorithm \ref{algorithm:IR-push-pull}, we set $\lambda_0:=0.1$. We employ the tuning rules according to Corollary \ref{cor:rate_unconstr_nonstrongly}, while a constant step-size is used for the Push-Pull method. We generate $H_i$ and $z_i$ randomly and choose $m=10$, $n=20$, and $d=1$. We generate matrices $\mathbf{R}$ and $\mathbf{C}$ from the same underlying graph with two different directed graphs (see Figure \ref{fig:comparison with push-pull}). We use $\mathbf{R}=\mathbf{I}-\tfrac{1}{2d_{in}^{\max}}\mathbf{L}_{\mathbf{R}}$ where $\mathbf{L}_{\mathbf{R}}$ denotes the Laplacian matrix and $d_{in}^{\max}$ denotes the maximum in-degree. We use the same formula for $\mathbf{C}$ using maximum out-degree. 
\begin{table}[H]
\setlength{\tabcolsep}{3pt}
\centering
 \begin{tabular}{ c} 
\includegraphics[scale=.42, angle=0]{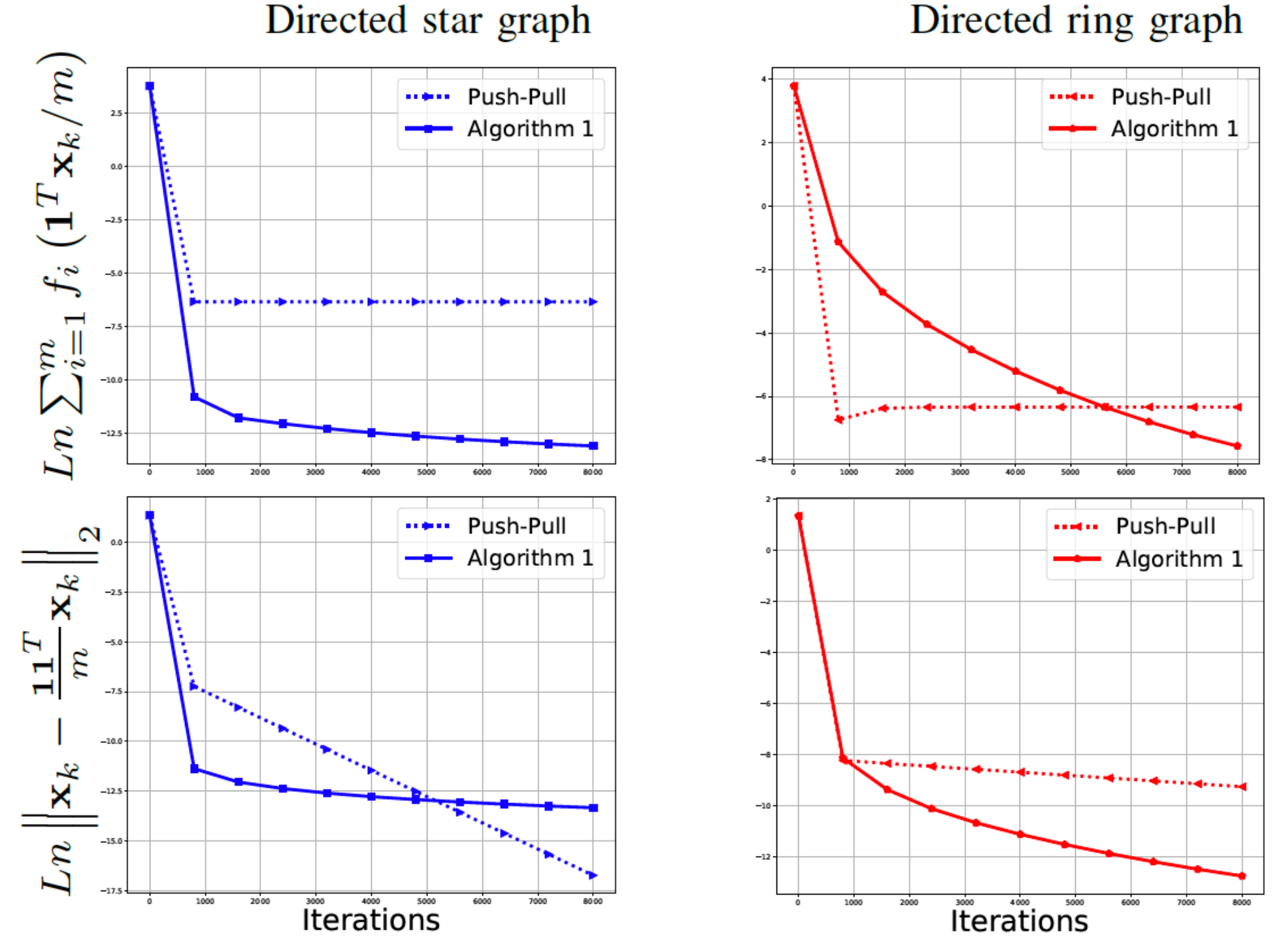}
\end{tabular}
\captionof{figure}{{Algorithm \ref{algorithm:IR-push-pull} vs. regularized Push-Pull algorithm under different choices of $\mathbf{R}$ and $\mathbf{C}$.\\}}
\vspace{-.2in}
\label{fig:comparison with push-pull}
\end{table}
\noindent \textit{Insights:} Figure \ref{fig:comparison with push-pull} shows the comparison of the two schemes. We compare objective function values and consensus violations. For the latter, we use the term $\left\|\mathbf{x}_k-\tfrac{\mathbf{1}\mathbf{1}^T}{m}\mathbf{x}_k\right\|_2$. In terms of the objective function value, Algorithm \ref{algorithm:IR-push-pull} performs significantly better both cases.

\noindent \textbf{(2) Distributed ill-conditioned linear inverse problems}: Here  $g(x) :=\textstyle\sum_{i=1}^m\|A_ix-b_i\|_2^2$ and $f(x):=\tfrac{1}{2}\|x\|^2_2$, where $A_i \in \mathbb{R}^{d\times n}$ and $b_i \in \mathbb{R}^d$ denote the locally known $i^{\textit{th}}$ block of the Toeplitz blurring operator and the given blurred image, respectively. Figure \ref{fig:debl} shows the progress of deblurring across the 9 agents over a directed ring graph.

\begin{table}[H]
\setlength{\tabcolsep}{3pt}
\centering
 \begin{tabular}{c  c  c } 
\begin{minipage}{.15\textwidth}
\includegraphics[scale=.15, angle=0]{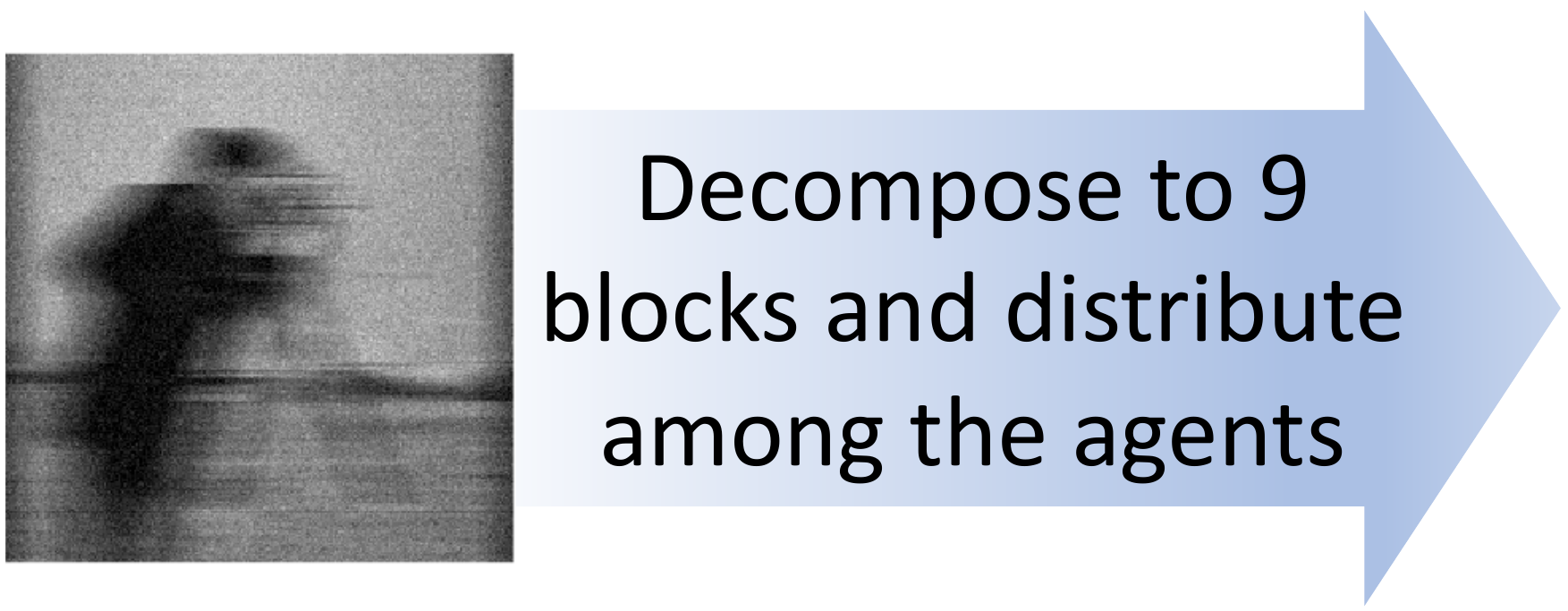}
\end{minipage}
&
\begin{minipage}{.15\textwidth}
\includegraphics[scale=.35, angle=0]{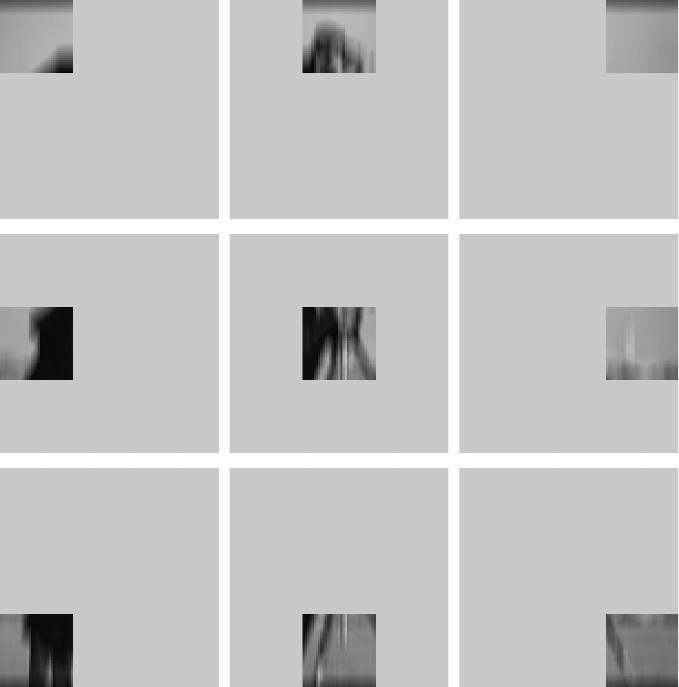}
\end{minipage}
&
\begin{minipage}{.15\textwidth}
\includegraphics[scale=.35, angle=0]{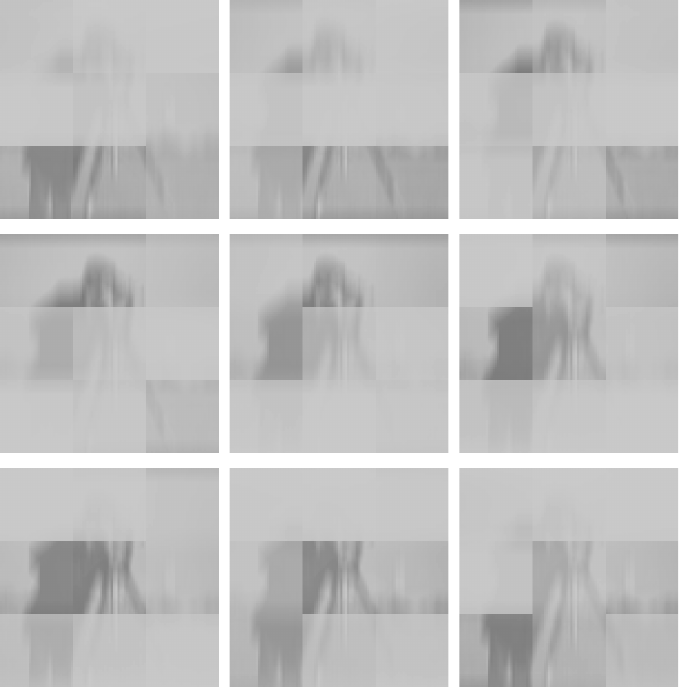}
\end{minipage}\\

\tiny{The blurred image} &\tiny{Initial distributed blocks} & \tiny{After $5$ iterations} 
\\

\begin{minipage}{.15\textwidth}
\includegraphics[scale=.35, angle=0]{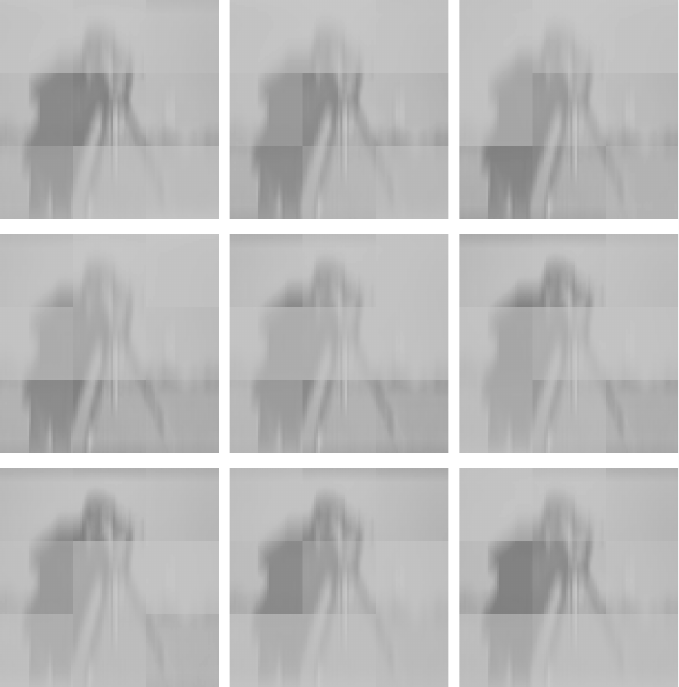}
\end{minipage}
	&
\begin{minipage}{.15\textwidth}
\includegraphics[scale=.35, angle=0]{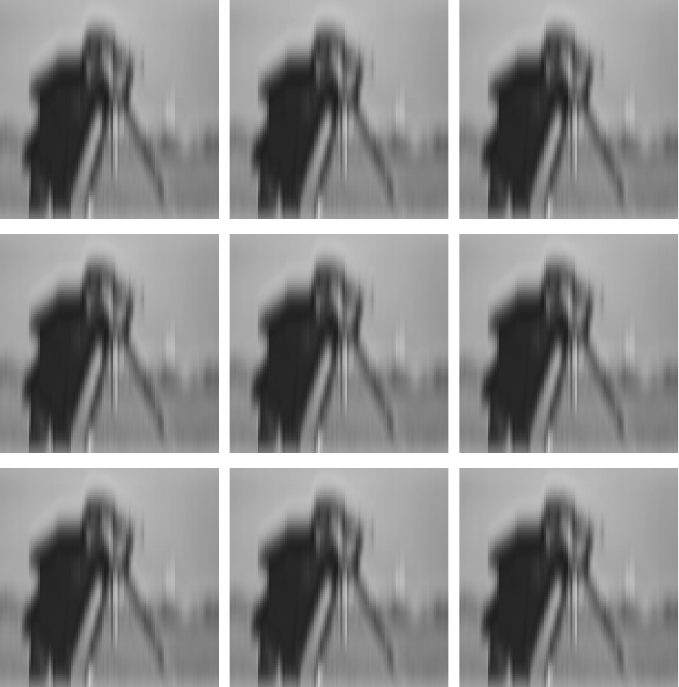}
\end{minipage}
	&
\begin{minipage}{.15\textwidth}
\includegraphics[scale=.35, angle=0]{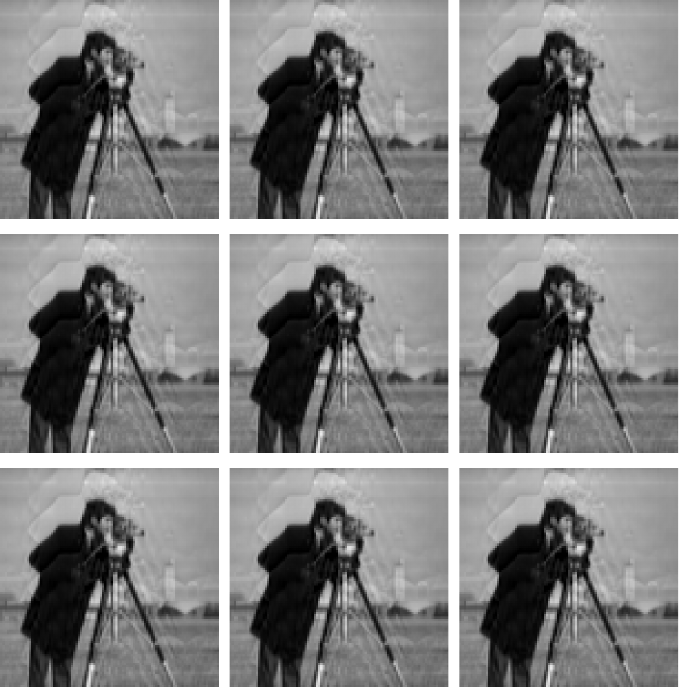}
\end{minipage}
\\
 \tiny{After $10$ iterations} & \tiny{After $10^2$ iterations}& \tiny{After $10^4$ iterations}
\end{tabular}
\captionof{figure}{Performance of IR-PushPull in distributed image deblurring using $9$ agents over a ring digraph}\label{fig:debl}
 \vspace{-.1in}
\label{fig:deblur}
\end{table}

\noindent \textbf{(3) Distributed linear SVM:} Consider a linear SVM where $\mathcal{D}\triangleq \{\left(u_\ell,v_\ell\right)\in \mathbb{R}^n\times \{-1,+1\}\mid \ell \in \mathcal{S}\}$ denotes the data set and $\mathcal{S}\triangleq \{1,\ldots,s\}$ denotes the index set. Let $\mathcal{S}$ be partitioned into $\mathcal{S}_{\tiny\mbox{train}}$ and $\mathcal{S}_{\tiny \mbox{test}}$ randomly. Let $\mathcal{S}_i$ denote the data locally known by agent $i$ where $\cup_{i=1}^m\mathcal{S}_i =\mathcal{S}_{\tiny\mbox{train}}$. Consider the following primal SVM model: 
\begin{equation}
\min_{x,b,z} \textstyle\sum_{i=1}^m\left(\tfrac{\eta}{2m}\|x\|_2^2+\textstyle\sum_{\ell \in \mathcal{S}_i}z_\ell \right) \  \text{s.t.}
 \begin{array}{@{}ll@{}}
    \scriptstyle{ v_\ell\left(x^Tu_\ell+b\right)\geq 1- z_\ell, }\\
    \scriptstyle{z_\ell\geq 0, \ \forall \ell \in \mathcal{S}_i, \ \forall i \in [m] ,}
  \end{array}
\end{equation} where $x \in \mathbb{R}^n, b \in \mathbb{R}, z \in \mathbb{R}^{|\mathcal{S}_{\tiny\mbox{train}}|}$, $\eta>0$. Figure \ref{fig:svm} shows the implementation of IR-PushPull on directed line and star graphs with $m:=10$ and  $\eta:=0.05$. 

\noindent \textit{Insights:} IR-PushPull performs very well compared to the centralized variant. This is examined both in terms of suboptimality and infeasibility metrics in different network topology settings. 

\begin{table}[t]
\setlength{\tabcolsep}{3pt}
\centering
 \begin{tabular}{ c} 
\includegraphics[scale=.25, angle=0]{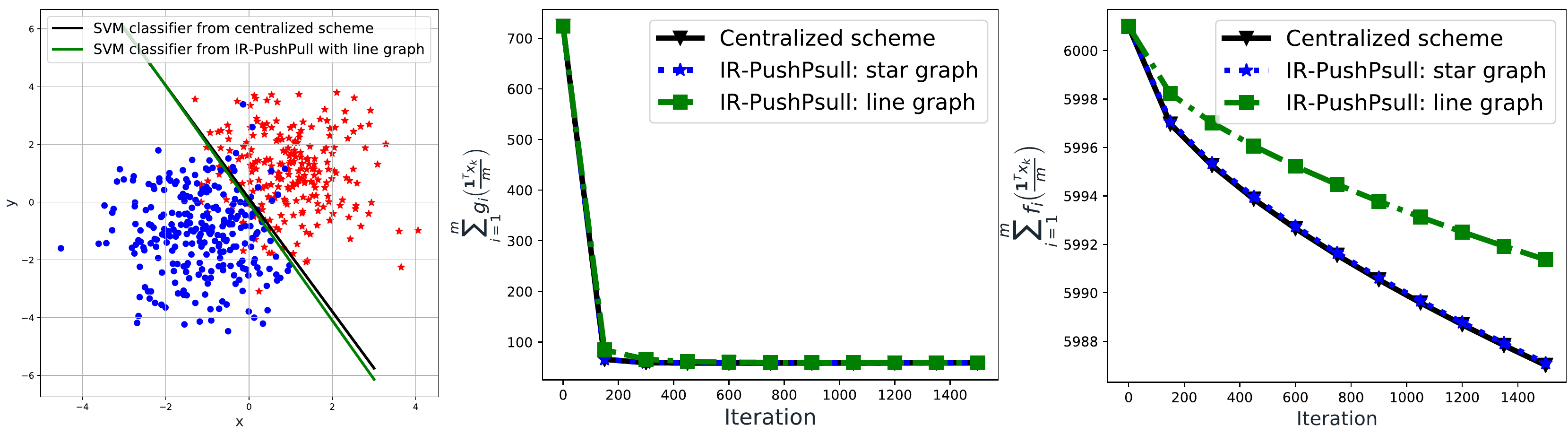}
\end{tabular}
\captionof{figure}{IR-PushPull vs. centralized scheme for  SVM: $10$ agents, $300$ training sample size, and $500$ testing sample size}
\vspace{-.2in}
\label{fig:svm}
\end{table}









\bibliographystyle{plain}
\bibliography{ref_fy,ref_allmypapers}

\begin{thebibliography}{10}

\bibitem{FarzadMostafaACC19}
M.~Amini and F.~Yousefian.
\newblock An iterative regularized incremental projected subgradient method for
  a class of bilevel optimization problems.
\newblock {\em Proceedings of the American Control Conference (ACC)}, pages
  4069--4074, 2019.

\bibitem{Aybat16}
N.~S. Aybat and E.~Y. Hamedani.
\newblock A primal-dual method for conic constrained distributed optimization
  problems.
\newblock {\em Advances in Neural Information Processing Systems}, pages
  5050--5058, 2016.

\bibitem{Aybat17}
N.~S. Aybat, Z.~Wang, T.~Lin, and S.~Ma.
\newblock Distributed linearized alternating direction method of multipliers
  for composite convex consensus optimization.
\newblock {\em IEEE Transactions on Automatic Control}, 63(1):5--20, 2017.

\bibitem{Bertsekas2016}
D.~P. Bertsekas.
\newblock {\em Nonlinear Programming: 3rd {E}dition}.
\newblock Athena Scientific, Bellmont, MA, 2016.

\bibitem{BertTsitBook}
D.~P. Bertsekas and J.~N. Tsitsiklis.
\newblock {\em Parallel and Distributed Computation: Numerical Methods}.
\newblock Athena Scientific, Belmont, MA, 1989.

\bibitem{BoydAdmm10}
S.~Boyd, N.~Parikh, E.~Chu, B.~Peleato, and J.~Eckstein.
\newblock Distributed optimization and statistical learning via the alternating
  direction method of multipliers.
\newblock {\em Foundations and Trends in Machine Learning}, 3(1):1--122, 2010.

\bibitem{Chang16}
T.~H. Chang.
\newblock A proximal dual consensus {ADMM} method for multi-agent constrained
  optimization.
\newblock {\em IEEE Transactions on Signal Processing}, 64(14):3719--3734,
  2016.

\bibitem{Chang14}
T.~H. Chang, A.~Nedi\'c, and A.~Scaglione.
\newblock Distributed constrained optimization by consensus-based primal-dual
  perturbation method.
\newblock {\em IEEE Transactions on Automatic Control}, 59(6):1524--1538, 2014.

\bibitem{Duarte14}
M.~F. Duarte and Y.~H. Hu.
\newblock Vehicle classification in distributed sensor networks.
\newblock {\em Journal of Parallel and Distributed Computing}, 64(7):826--838,
  2014.

\bibitem{Duchi12}
J.~C. Duchi, P.~L. Bartlett, and Martin~J. Wainwright.
\newblock Randomized smoothing for stochastic optimization.
\newblock {\em SIAM Journal on Optimization (SIOPT)}, 22(2):674--701, 2012.

\bibitem{facchinei02finite}
F.~Facchinei and J.-S. Pang.
\newblock {\em Finite-dimensional Variational Inequalities and Complementarity
  Problems. {V}ols. {I,II}}.
\newblock Springer Series in Operations Research. Springer-Verlag, New York,
  2003.

\bibitem{Groot74}
M.~H.~De Groot.
\newblock Reaching a consensus.
\newblock {\em Journal of the American Statistical Association},
  69(345):118--121, 1974.

\bibitem{FastDO14}
D.~Jakoveti\'c, J.~Xavier, and J.~M. Moura.
\newblock Fast distributed gradient methods.
\newblock {\em IEEE Transactions on Automatic Control}, 59(5):1131--1146, 2014.

\bibitem{FarzadHarshalACC19}
H.~Kaushik and F.~Yousefian.
\newblock A randomized block coordinate iterative regularized subgradient
  method for high-dimensional ill-posed convex optimization.
\newblock {\em Proceedings of the American Control Conference (ACC)}, pages
  3420--3425, 2019.

\bibitem{HarshalFarzadOptVI2021}
H.~D. Kaushik and F.~Yousefian.
\newblock A method with convergence rates for optimization problems with
  variational inequality constraints.
\newblock 2021.
\newblock {arXiv preprint:}
  \href{https://arxiv.org/pdf/2007.15845.pdf}{https://arxiv.org/pdf/2007.15845.pdf}.

\bibitem{Lee16}
S.~Lee and A.~Nedi\'c.
\newblock Asynchronous gossip-based random projection algorithms over networks.
\newblock {\em IEEE Transactions on Automatic Control}, 61(4):953--958, 2016.

\bibitem{Ling14}
Q.~Ling and A.~Ribeiro.
\newblock Decentralized dynamic optimization through the alternating direction
  method of multiplier.
\newblock {\em IEEE Transactions on Signal Processing}, 62(5):1185--1197, 2014.

\bibitem{Lobel11}
I.~Lobel, A.~Ozdaglar, and D.~Feijer.
\newblock Distributed multi-agent optimization with state-dependent
  communication.
\newblock {\em Mathematical Programming}, 129(2):255–--284, 2011.

\bibitem{Lopes07}
C.~G. Lopes and A.~H. Sayed.
\newblock Distributed processing over adaptive networks.
\newblock {\em In 2007 9th International Symposium on Signal Processing and Its
  Applications}, pages 1--3, 2007.

\bibitem{Lorenzo16}
P.~Di Lorenzo and G.~Scutari.
\newblock Distributed nonconvex optimization over time-varying networks.
\newblock {\em IEEE International Conference on Acoustics, Speech and Signal
  Processing (ICASSP)}, 2016.

\bibitem{Nunez15}
D.~Mateos-Nunez and J.~Cortes.
\newblock Distributed subgradient methods for saddle-point problems.
\newblock {\em IEEE Conference on Decision and Control (CDC)}, 2015.

\bibitem{NedichAlex15}
A.~Nedi\'c and A.~Olshevsky.
\newblock Distributed optimization over time-varying directed graphs.
\newblock {\em IEEE Transactions on Automatic Control}, 60(3):601--615, 2015.

\bibitem{NedichAlex16}
A.~Nedi\'c and A.~Olshevsky.
\newblock Stochastic gradient-push for strongly convex functions on
  time-varying directed graphs.
\newblock {\em IEEE Transactions on Automatic Control}, 61(12):3936--3947,
  2016.

\bibitem{Diging17}
A.~Nedi\'c, A.~Olshevsky, and W.~Shi.
\newblock Achieving geometric convergence for distributed optimization over
  time-varying graphs.
\newblock {\em SIAM Journal on Optimization}, 27(4):2597--2633, 2017.

\bibitem{Nedich09}
A.~Nedi\'c and A.~Ozdaglar.
\newblock Subgradient methods for saddle-point problems.
\newblock {\em Journal of Optimization Theory and Applications},
  142(1):205--228, 2009.

\bibitem{pushi2_2020}
S.~Pu and A.~Nedi\'c.
\newblock Distributed stochastic gradient tracking methods.
\newblock {\em Mathematical Programming}, 2020.

\bibitem{pushi1_2020}
S.~Pu, W.~Shi, J.~Xu, and A.~Nedi\'c.
\newblock Push-pull gradient methods for distributed optimization in networks.
\newblock {\em IEEE Transactions on Automatic Control}, 2020.

\bibitem{QuLi18}
G.~Qu and N.~Li.
\newblock Harnessing smoothness to accelerate distributed optimization.
\newblock {\em IEEE Transactions on Control of Network Systems},
  5(3):1245--1260, 2018.

\bibitem{Sun19}
G.~Scutari and Y.~Sun.
\newblock Distributed nonconvex constrained optimization over time-varying
  digraphs.
\newblock {\em Mathematical Programming}, 176:497--544, 2019.

\bibitem{EXTRA15}
W.~Shi, Q.~Ling, G.~Wu, and W.~Yin.
\newblock {EXTRA}: an exact first-order algorithm for decentralized consensus
  optimization.
\newblock {\em SIAM Journal on Optimization}, 25(2):5944--966, 2015.

\bibitem{Shi14}
W.~Shi, Q.~Ling, K.~Yuan, G.~Wu, and W.~Yin.
\newblock On the linear convergence of the {ADMM} in decentralized consensus
  optimization.
\newblock {\em IEEE Transactions on Signal Processing}, 62(7):1750--1761, 2014.

\bibitem{Srivastava11}
K.~Srivastava and A.~Nedi\'c.
\newblock Distributed asynchronous constrained stochastic optimization.
\newblock {\em IEEE Journal of Selected Topics in Signal Processing},
  5(4):772–--90, 2011.

\bibitem{Touri17}
T.~Tatarenki and B.~Touri.
\newblock Non-convex distributed optimization.
\newblock {\em IEEE Transactions on Automatic Control}, 62(8):3744--3757, 2017.

\bibitem{Tsit84}
J.~N. Tsitsiklis.
\newblock {\em Problems in Decentralized Decision Making and Computation}.
\newblock PhD thesis, Dept. of Electrical Engineering and Computer Science,
  Massachussetts Institute of Technology, 1984.

\bibitem{Tsit86}
J.~N. Tsitsiklis, D.~P. Bertsekas, and M.~Athans.
\newblock Distributed asynchronous deterministic and stochastic gradient
  optimization algorithms.
\newblock {\em IEEE Transactions on Automatic Control}, 31(9):803–--812,
  1986.

\bibitem{Wei12}
E.~Wei and A.~Ozdaglar.
\newblock Distributed alternating direction method of multipliers.
\newblock {\em IEEE Conference on Decision and Control (CDC)}, pages
  5445--5450, 2012.

\bibitem{Wei13}
E.~Wei and A.~Ozdaglar.
\newblock On the {O}$(1/k)$ convergence of asynchronous distributed alternating
  direction method of multipliers.
\newblock {\em In 2013 IEEE Global Conference on Signal and Information
  Processing}, pages 551--554, 2013.

\bibitem{FarzadMathProg17}
F.~Yousefian, A.~Nedi\'c, and U.~V. Shanbhag.
\newblock On smoothing, regularization, and averaging in stochastic
  approximation methods for stochastic variational inequality problems.
\newblock {\em Mathematical Programming}, 165(1):391--431, 2017.

\bibitem{FarzadSIOPT20}
F.~Yousefian, A.~Nedi\'c, and U.~V. Shanbhag.
\newblock On stochastic and deterministic quasi-{N}ewton methods for
  nonstrongly convex optimization: Asymptotic convergence and rate analysis.
\newblock {\em SIAM Journal on Optimization}, 30(2):1144--1172, 2020.

\end{thebibliography}

\end{document}